\documentclass[a4paper,11pt]{article}
\usepackage[latin1]{inputenc}
\usepackage{amsfonts}
\usepackage{amsthm}
\usepackage{amsmath}
\usepackage{amsfonts}
\usepackage{latexsym}
\usepackage{amssymb}
\usepackage{dsfont}
\usepackage{tocloft}
\usepackage[usenames]{color}

\usepackage{calc}

\usepackage[bookmarks,pdfpagelabels,linkbordercolor={0.1 0.8 0.4}]{hyperref}

\newtheorem{teo}{Theorem}[section]
\newtheorem*{teo*}{Theorem}
\newtheorem{lem}[teo]{Lemma}
\newtheorem{cor}[teo]{Corollary}
\newtheorem{prop}[teo]{Proposition}

\theoremstyle{definition}
\newtheorem{fed}[teo]{Definition}

\theoremstyle{remark}
\newtheorem{rem}[teo]{Remark}

\newtheorem{nota}[teo]{Notation}

\newtheorem{exa}[teo]{Example}

\def\Z{\mathbb{Z}}
\def\R{\mathbb{R}}
\def\C{\mathbb{C}}

\def\cA{\mathcal{A}}
\def\cB{\mathrm{B}}

\def\cD{\mathcal{D}}
\def\cE{\mathcal{E}}
\def\cF{\mathfrak{F}}
\def\cG{\mathcal{G}}
\def\cH{\mathcal{H}}
\def\cI{\mathcal{I}}

\def\cO{\mathcal{O}}

\def\cM{\mathcal{M}}

\def\cU{\mathcal{U}}
\def\cV{\mathcal{V}}
\def\cW{\mathcal{W}}
\def\cX{\mathcal{X}}
\def\cY{\mathcal{Y}}
\def\cZ{\mathcal{Z}}

\def\Z{\mathcal{Z}}

\def\fu{\mathfrak{u}}
\def\fg{\mathfrak{g}}
\def\fH{\mathfrak{H}}
\def\fK{\mathfrak{k}}

\def\fo{\mathfrak{o}}

\def\Ad{\mathrm{Ad}}
\def\ad{\mathrm{ad}}
\def\GL{\mathrm{GL}}
\def\rG{\mathrm{G}}
\def\O{\mathrm{O}}
\def\U{\mathrm{U}}
\def\UB{\mathrm{U}_{\mathrm{Bog}}}
\def\UI{\mathrm{U}_{\mathrm{Imp}}}
\def\DF{\mathcal{D}}
\def\ui{\mathfrak{u}_{\mathrm{Imp}}}

\def\fZ{\mathfrak{z}}

\def\ubg{\mathfrak{u}^{\Gamma}_{\mathrm{Bog}}}

\def\fN{\mathfrak{n}}
\def\fP{\mathfrak{p}}
\def\fG{\mathfrak{g}}

\def\d1{\mathbf{1}}
\def\dU{\mathds{U}}

\def\o{\omega}
\def\Z{\cZ_{\mathrm{qf}}}
\def\dN{\mathds{N}}
\def\dH{\mathds{H}}

\def\diag{\mbox{diag}}
\def\Ad{\mbox{Ad}}
\def\ad{\mbox{ad}}
\def\ker{\mbox{ker}}
\def\ind{\mathrm{ind}}

\def\nL{\mathrm{\mathbf{L}}}

\def\tde{\tilde{\delta}_\Gamma}

\def\noi{\noindent}

\def\bdem{\begin{proof}}
\def\edem{\renewcommand{\qed}{\hfill $\blacksquare$}
\end{proof}}

\DeclareMathOperator\ub{\mathfrak{u}_{\mathrm{Bog}}}

\DeclareMathOperator{\Tr}{Tr}

\DeclareMathOperator{\ran}{ran}

\DeclareMathOperator{\res}{res}

\newcommand{\PI}[2]{\left\langle #1 , #2 \right\rangle}

\title{{\Large Homogeneous spaces in  Hartree-Fock-Bogoliubov theory}}
\author{Claudia D. Alvarado, Eduardo Chiumiento}
\date{}

\begin{document}

\maketitle

\begin{abstract}
We study   the   action of Bogoliubov transformations on admissible generalized one-particle density matrices  arising in  Hartree-Fock-Bogoliubov theory.  We show that  the  orbits of this action are reductive homogeneous spaces, and we give several equivalences that characterize when they are embedded submanifolds of natural ambient spaces. We use Lie theoretic arguments to prove that these orbits admit  an invariant symplectic form. If, in addition, the operators in the orbits have finite spectrum, or infinite spectrum and trivial kernel, then we obtain that  the orbits are actually    K\"ahler homogeneous spaces.
\end{abstract}


\tableofcontents

\medskip

\noi {\it 2020 MSC.}   Primary 22E65  - Secondary 53C15, 53Z05, 81V55


\medskip

\noi \textit{Keywords.} Generalized one-particle density matrix; Bogoliubov transformation; homogeneous space; embedded submanifold; invariant symplectic form; K\"ahler homogeneous space.


\section{Introduction}

 Hartree-Fock-Bogoliubov (HFB) theory, also known as generalized Hartree-Fock theory, is a relevant tool for understanding fermionic  many-body  quantum systems.  It is a generalization of the traditional Hartree-Fock theory  for systems in which the particle number is not assumed to be conserved. Although it has been studied earlier  in the physics literature \cite{BCS57, B59, V61}, pioneering work on rigorous mathematical aspects of HFB theory was done by Bach, Lieb and Solovej \cite{BLS94}.
In the present paper we investigate the main objects in HFB theory from the point of view of 
infinite-dimensional geometry by using Lie groups  modeled on Banach spaces and  its homogeneous spaces. 

\medskip

We now briefly explain the motivation and framework related to HFB theory (see Appendix \ref{HFB} for precise definitions and details). Let $\fH$ be the one-particle Hilbert space, and let $\cF=\cF[\fH]$ be its associated fermionic Fock space.
 Let $\cZ$ be the set of all normal states on $\cF$. 
%
Consider a Hamiltonian $\dH$ bounded from below on $\cF$ describing the dynamics of a   fermionic system.     It is of physical importance to compute the following 
$$
E_{gs}:=\inf \{ \omega(\dH)  :   \omega \in \cZ \}. 
$$
This is known as the total ground state energy in the grand canonical ensemble.  Approximation methods become relevant for such task, which turns out to be impossible in most cases.   In HFB theory the total ground state energy is approximated by taking as trial states the \textit{quasi-free states with finite particle number}, which is a class of states that obeys Wick's theorem. 
Let us denote by $\Z \subseteq \cZ$ the set of all   quasi-free states with finite particle number on $\cF$.
The  HFB energy is then defined by
\begin{equation*}
E_{HFB}:=\inf\{ \omega(\dH) :  \omega \in \Z \}.
\end{equation*}
Clearly, we have the upper bound $E_{gs} \leq E_{HFB}$. Now a key point, which is convenient for the calculus of variations, is that one can  rewrite the HFB in terms of an infimum over a convex set of operators acting on $\fH \oplus \fH$. Denote by $ \cB_1(\fH)$ and $\cB_2(\fH)$ the trace-class  and Hilbert-Schmidt operators on $\fH$, respectively. The set $\Z$ is indeed in a bijective correspondence with the set of all \textit{admissible generalized one-particle density matrices (g1-pdm)} 
 \begin{equation*}
\cD:=\left\{  \begin{pmatrix} \gamma & \alpha   \\   \alpha^*   & \d1 - \bar{\gamma}  \end{pmatrix} \in \cB(\fH \oplus \fH) :   0 \leq \begin{pmatrix} \gamma & \alpha   \\   \alpha^*   & \d1 - \bar{\gamma}  \end{pmatrix} \leq \d1_{\fH\oplus \fH}, \, \gamma=\gamma^* \in  \cB_1(\fH), \, \alpha^T=-\alpha   \right\},
 \end{equation*}
where we write  $\bar{\gamma}=I_0 \alpha I_0$, $\bar{\alpha}=I_0 \alpha I_0$, $\alpha^T=\bar{\alpha}^*$, and $I_0$ is a fixed conjugate-linear isometry on $\fH$. 
Thus, each  $\Gamma=\Gamma[\gamma, \alpha] \in \cD$,  is described by an operator $\gamma \in \cB_1(\fH)$, $0 \leq \gamma \leq \d1$, called the \textit{one-particle density matrix}, and another operator  $\alpha$, called the \textit{pairing matrix}. We observe that one can check that $\alpha \in \cB_2(\fH)$ by elementary computations. Suppose that $\omega \mapsto \Gamma_\omega$ denotes the correspondence between $\Z$ and $\cD$, and define the HFB functional by
$\cE(\Gamma_\omega):=\omega(\dH)$, 
then the HFB energy can be rewritten in terms of admissible g1-pdms as
\begin{equation}\label{HFB min}
E_{HFB}=\inf\{  \cE(\Gamma) : \Gamma \in \cD \}.
\end{equation}
Let $\U(\fH \oplus \fH)$ be the unitary group acting on $\fH \oplus \fH$. An useful  fact in HFB theory is that  the so-called \textit{Bogoliubov transformations satisfying the Shale-Stinespring condition} given by
$$
\UB  :=\left\{ \begin{pmatrix}  u  &   v   \\   \bar{v} &  \bar{u}  \end{pmatrix} \in \U(\fH \oplus \fH):  v \in \cB_2(\fH)  \right\}
$$
defines an action by conjugation 
$$\UB \times \cD \to \cD,  \, \, \, U \cdot \Gamma=U\Gamma U^*, \, \, \,  U \in \UB,  \, \Gamma \in \cD.
$$
The bijection between quasi-free states and g1-pdms turns out to be an equivariant map, where in the set of quasi-free states one has the action of the subgroup of the unitary group on $\cF$  known as the unitary implementers. 



 
Appart from the previously mentioned work \cite{BLS94}, other    mathematical results in HFB theory include the following ones. In treating the existence of minimizers in \eqref{HFB min} a serious difficulty one has to face is the lack of weak lower semicontinuity of the HFB functional.  Remarkably,   the existence of minimizers was established by  Lenzmann and Lewin \cite{LL10} for  pseudorelativistic fermions interacting with Newtonian gravitational forces. On the other hand,  the  occurrence of pairing  is a phenomena to be understood within HFB theory. This amounts to know  when there exist minimizers of the form $\Gamma=\Gamma[\gamma, \alpha]$, $\alpha \neq 0$. Bach, Fr\"ohlich and  Jonsson  \cite{VFJ09} developed a simplification of HFB theory under certain abstract assumptions, to find that the pairing matrix can be expressed in terms of the one-particle matrix.  
In \cite{BBKM14}  a generalized Lieb's variational principle was proved, which  roughly stated, means that in \eqref{HFB min} one can only consider the infimum over those admissible g1-pdms that are projections ($\Gamma^2=\Gamma$). Finally, we refer to \cite{LS14} for a numerical approach to the pairing and other problems in HFB theory,    and the recent survey \cite{B22} for  the  aspects mentioned in this paragraph  and others in relation to HFB  theory. 

\medskip

Throughout this paper, $\fH$ is assumed to be an infinite-dimensional Hilbert space. 
We study the geometric structure of  orbits given by the conjugacy action of  Bogoliubov transformations on admissible g1-pdms: for $\Gamma \in \cD$,
\begin{equation*}
\cO(\Gamma) :=\{ U\Gamma U^* : U \in \UB  \}.
\end{equation*}
We will see that these orbits are concrete examples of infinite-dimensional homogeneous spaces 
related to quantum mechanics. Geometric structures such  as symplectic forms, complex or K\"ahler homogeneous structures that are pervasive in classical mechanics will be constructed for orbits of g1-pdms by using Lie algebraic arguments. Such kind of arguments are based on general results in the existing literature of infinite-dimensional homogeneous spaces (see for instance   \cite{B06, N04, OR03, Up85}).
 Let us point out that the geometry of homogeneous spaces in the traditional Hartree-Fock was considered in \cite{CM12}; meanwhile in  the physics literature,  the above  orbits of g1-pdms were investigated  in \cite{R81} under the  assumption that they can be modeled as finite-dimensional manifolds.   

In Section \ref{Differential} we  study the reductive homogeneous space structure of orbits of admissible g1-pdms and quasi-free states with finite particle number. We first recall some useful results of the Lie group structure of the Bogoliubov transformations that are scattered in the literature on both Lie groups and HFB theory. In fact,  $\UB$ belongs to the class of restricted Lie groups studied by Neeb \cite{N},  and its Lie algebra is identified with
$$
\ub:= \left\{ \, \begin{pmatrix} x_1  &  x_2  \\   \bar{x}_2  &  \bar{x}_1   \end{pmatrix} \in \cB(\fH\oplus \fH) :  x_1=-x_1^*, \, x_2=-x_2^T \in  \cB_2(\fH)  \right\},
$$
which has the following restricted norm 
$$
\left\|\begin{pmatrix} x_1  &  x_2  \\   \bar{x}_2  &  \bar{x}_1   \end{pmatrix} \right\|_{\res}:=2 \max\{  \|  x_1\| , \| x_2 \|_2 \}.
$$
We show that the orbits  $\cO(\Gamma)\simeq \UB / \UB^\Gamma$ can be endowed with the structure of  reductive homogeneous space, where $ \UB^\Gamma$ is the isotropy group at $\Gamma \in \cD$. We also prove that $\cO(\Gamma)$ is connected if and only if $\frac{1}{2}$ is not an eigenvalue of $\Gamma$. Otherwise, $\cO(\Gamma)$ has two connected components (see Theorem \ref{g1pdm smooth homog spaces}). To construct the infinite-dimensional homogeneous space structure one has  to find an accessible expression of the isotropy group $\UB^\Gamma$  to show that its Lie algebra is a complemented subspace of   $\ub$. This can be done thanks to a well-known diagonalization result for g1-pdms by Bogoliubov transformations \cite{BLS94}, which implies that each orbit has a diagonal g1-pdm. As a direct  consequence of the mentioned correspondence between g1-pdms and quasi-free states, we obtain that orbits of quasi-free states with finite particle number are reductive homogeneous spaces. We observe that the projection $P_-$ on $\fH \oplus \fH$ defined by $P_-(f,g)=(0,g)$ is a g1-pdm, and the orbit $\cO(P_-)$ is actually  the isotropic restricted Grassmannian related to loop groups (see \cite{PS}).

In Section \ref{submanifolds} we give several equivalent characterizations to guarantee that orbits of g1-pdms are  embedded submanifolds of natural affine spaces. Besides the extrinsic geometric structure of the orbits as homogeneous spaces, which are in this way equipped with the quotient topology, it is desirable to give 
a more intrinsic type of manifold structure on  them. The following inclusion is not difficult to check
$$
\cO(\Gamma) \subseteq i \ub  + \begin{pmatrix}   0   &   0   \\ 0  &   \d1  \end{pmatrix}:=i \ub + P_- .
$$ 
This gives a relative topology on g1-pdms inherited from $i \ub + P_-$, whose  restricted metric is defined by $d(\Gamma_0, \Gamma_1)=\| \Gamma_0 - \Gamma_1\|_{\res}$. We prove that the following conditions are equivalent: tangent spaces of $\cO(\Gamma)$ are closed in $i \ub$,  the quotient topology and the relative  topology  coincide in $\cO(\Gamma)$ , and $\Gamma$ has finite spectrum. Furthermore, these are also equivalent to being $\cO(\Gamma)$  an embedded submanifold of $i \ub + P_-$ (Theorem \ref{submanifold}).  In the process, we discuss  derivations induced by g1-pdms defined as follows: for each $\Gamma \in \cD$, set  $\delta_\Gamma:\ub \to \ub$, $\delta_\Gamma(X)=[i\Gamma, X]$,  for all $X \in \ub$. These, in turn, are useful in the construction of continuous local cross sections for the action that allow us to compare the two mentioned topologies.
This circle of ideas has recently been studied  for unitary orbits in operator ideals by Belti\c t$\breve{\text{a}}$ and Larotonda (see \cite{BL23} and the references therein). Also our results on the topology of g1-pdms  might be considered as  a natural extension of similar ones related to one-particle density matrices (see \cite{Bo04,GKMS18}).  
In contrast to the aforementioned works that deal with unitary invariant metrics, notice that  the action of Bogoliubov transformations is not invariant for the previous restricted metric. 

In  Section \ref{symplectic complex kahler} we construct an invariant symplectic  form on the orbits. We introduce for each $\Gamma \in \cD$ the continuous 2-cocycle $s_\Gamma:\ub \times \ub \to \R$ defined by
$$
s_\Gamma(X,Y):=   \Tr(X[i\Gamma,Y]).
$$ 
Then a standard construction can be applied to obtain a weakly symplectic homogeneous space $(\cO(\Gamma), \omega)$, where $\omega=\Sigma(s_\Gamma)$ is the symplectic form induced by $s_\Gamma$. In the case in which $\Gamma$ has finite spectrum,  we  show that $\omega$ is a strong symplectic form. 
Furthermore, we prove in Theorem \ref{kahler structures} that $(\cO(\Gamma), \omega)$ is a K\"ahler homogeneous space in the following two cases: when the spectrum of $\Gamma$ is finite, or when 
the spectrum of $\Gamma$ is infinite and $\ker(\Gamma)=\{ 0 \}$. This is shown by using the notion of K\"ahler polarizations. In this regard, we refer to  \cite{B06, N04} for general results  in the infinite-dimensional setting, and to \cite{B05a, BRT05, N04} for interesting examples. 

 We include an appendix  to make the exposition more self-contained.  It contains notation and necessary results on HFB theory, and it treats the needed geometric structures in the setting of Banach manifolds.

\section{Reductive homogeneous spaces}\label{Differential}



\subsection{Lie group structure of Bogoliubov transformations}

We assume that $\fH$ is an infinite-dimensional separable complex Hilbert space, 
 and we consider the group of  Bogoliubov transformations satisfying the Shale-Stinespring condition, namely 
\begin{align*}
\UB  :=\left\{ \begin{pmatrix}  u  &   v   \\   \bar{v} &  \bar{u}  \end{pmatrix} \in \U(\fH \oplus \fH):  v \in \cB_2(\fH)  \right\}.
\end{align*} 
The conjugate-linear isometry $I_0$ that defines $\bar{x}=I_0xI_0$ for every $x \in \cB(\fH)$ is associated to a fixed orthonormal basis $\{ \varphi_k \}_{k\geq 1}$ of $\fH$. This means that $I_0$ is defined by $I_0(\sum_{k \geq 1} \mu_k \varphi_k)=\sum_{k \geq 1}\bar{\mu}_k \varphi_k$, for every sequence $\{ \mu_k \}_{k \geq 1} \in \ell^2(\C)$. 

We begin by recalling several topological and geometric properties of the group $\UB$. Many of them are indeed well-known in the Lie groups literature; our presentation intend  to relate them with the Hartree-Fock-Bogoliubov framework in the most convenient way to our purposes. Among other aspects we observe the relation of Bogoliubov transformations with an orthogonal group that appears in the Clifford algebra formulation. Also it turns out that $\UB$ fits into the class of the so-called \textit{restricted Lie groups}. Our main reference is the work  of Neeb  (\cite{N}), where the reader can find a detailed account about these groups and its Lie algebras.

\begin{rem}\label{rem Banach alg-restricted orthogonal-isomorphism} 
We recall a Banach algebra and two Lie groups  related to $\UB$.

\medskip

\noi $i)$ \textit{The restricted algebra.} 
Let $P_{\pm}$ be the orthogonal projections on $\fH \oplus \fH$ defined by $P_+(f,g)=(f,0)$ and $P_-(f,g)=(0,g)$, for all $f,g \in \fH$. We set
$$
D:=P_+ - P_-=\begin{pmatrix} \d1  &   0 \\  0   & -\d1  \end{pmatrix}.
$$
Consider the following \textit{restricted algebra associated with $D$ and the Hilbert-Schmidt operators}:
\begin{align*}\label{rest alg}
\cB_2(\fH \oplus \fH, D) & :=\{  X \in \cB(\fH \oplus \fH) :  [X,D] \in \cB_2(\fH \oplus \fH)  \} \\
&  = \left\{   \begin{pmatrix}  x_{11}   &  x_{12} \\  x_{21}  &  x_{22}   \end{pmatrix} \in \cB(\fH \oplus \fH) :  x_{12}, \, x_{21} \in \cB_2(\fH)         \right\}.
\end{align*}
This is a complex Banach $*$-algebra equipped with the norm 
\begin{equation}\label{res norm}
\| X  \|_{\res}:= 2 \max \{ \|x_{11}\|, \, \| x_{22}\|, \, \|x_{12}\|_2 , \, \| x_{21}\|_2 \},
\end{equation}
where $\| \, \cdot \,\|$ and $\| \, \cdot \,\|_2$ are the operator norm and the Hilbert-Schmidt norm, respectively. 
We refer to this norm as the \textit{restricted norm}. 
The factor $2$ is added to get the submutiplicative property of this norm, which can be checked by elementary estimates with the operator and 
Hilbert-Schmidt norms. Notice that $\UB \subseteq \cB_2(\fH \oplus \fH, D)$. We remark that multiplication by unitaries in $\UB$ is not isometric for the restricted norm, i.e. in general $\| UX\|_{\res} \neq  \| X   \|_{\res}$ and $\| XU\|_{\res} \neq  \| X   \|_{\res}$,
for $U \in \UB$ and $X \in \cB_2(\fH \oplus \fH, D)$.

\medskip

\noi $ii)$ \textit{A restricted general linear group.}  
Recalling that $I_0$ is the  conjugation associated to the fixed orthonormal basis $\{  \varphi_k\}_{k \geq 1}$, we put 
\begin{equation}\label{conjugation I}
I:=\begin{pmatrix}   0  &  I_0   \\  I_0   &  0  \end{pmatrix}.
\end{equation}
Clearly, $I$ is a conjugation acting on $\fH\oplus \fH$. Also note that $I^*=I$. As $I$ is conjugate-linear this means that $\PI{(f_1,f_2)}{I (g_1,g_2)}=\PI{(g_1,g_2)}{I(f_1,f_2)}$, for all $(f_1,f_2), (g_1 , g_2) \in  \fH \oplus \fH$. 
Let $\GL(\fH\oplus \fH)$ denotes the group of linear invertible operators on $\fH \oplus \fH$. As a particular case of \cite[Def. III.3]{N} obtained by taking the aforementioned conjugation $I$ and self-adjoint operator $D$, we consider the following \textit{restricted general linear group}
\begin{align*}
\GL_2(\fH \oplus \fH,I,D) & :=\{ G \in \GL(\fH \oplus \fH) :   G^{-1}=IG^*I, \, [G,D] \in \cB_2(\fH \oplus \fH)  \} \\
& =\left\{  \begin{pmatrix}  g_{11}   &  g_{12} \\  g_{21}  &  g_{22}   \end{pmatrix} \in \GL(\fH \oplus \fH) :  \begin{pmatrix}  g_{11}   &  g_{12} \\  g_{21}  &  g_{22}   \end{pmatrix}^{-1}=\begin{pmatrix}  g_{22}^T   &   g_{12}^T  \\   g_{21}^T   &  g_{11}^T \end{pmatrix}, \, g_{12}, \, g_{21}  \in \cB_2(\fH)      \right\}
\end{align*}
This group carries a Lie group structure endowed with the restricted norm. 
As a direct consequence of the matrix description of $\GL_2(\fH \oplus \fH,I,D)$, notice that 
\begin{align}
\UB & =  \{ \, U \in \mathrm{U}(\fH \oplus \fH)  :     U= I U I  ,    \, [U,D] \in \cB_2(\fH \oplus \fH)    \} \nonumber \\
& =\U(\fH\oplus \fH) \cap \GL_2(\fH \oplus \fH,I,D).  \label{ubog and g}
\end{align}
The Lie algebra of $\GL_2(\fH \oplus \fH,I,D)$ has the usual commutator of operators as Lie bracket, and it is given by
\begin{align}
\mathfrak{gl}_2(\fH \oplus \fH,I,D) & :=\{ X \in \cB(\fH \oplus \fH) :   X=-IX^*I, \, [X,D] \in \cB_2(\fH \oplus \fH)  \}  \nonumber  \\
& = \left\{  \begin{pmatrix} x_{11}  & x_{12}   \\   x_{21}   &   -x_{11}^T    \end{pmatrix} \in \cB(\fH \oplus \fH) :   x_{12}=-x_{12}^T \in \cB_2(\fH), x_{21}=-x_{21}^T \in \cB_2(\fH) \right\}. \label{complexif Lie alg}
\end{align}
 Finally, we note that $\mathfrak{gl}_2(\fH \oplus \fH,I,D)$ is an involutive Lie algebra. This means that is a complex Lie algebra endowed with an involution $*$ such that $(X^*)^*=X$ and $[X,Y]^*=[Y^*,X^*]$. The involution is given by the usual  operator adjoint. In particular,
the involution determines a real form of $\fg$, which is  the real subalgebra defined by 
$$
\mathfrak{gl}_2(\fH \oplus \fH,I,D)_\R:=\{ X \in \mathfrak{gl}_2(\fH \oplus \fH,I,D) : X^*=-X\}.
$$  
The Lie algebra  $\mathfrak{gl}_2(\fH \oplus \fH,I,D)$ will be used  in Section  \ref{symplectic complex kahler} for  the construction of K\"ahler structures on orbits of g1-pdms.

\medskip

\noi $iii)$ \textit{The restricted orthogonal group.}
 We write $\fH^\R$ for the underlying real Hilbert space, and let $J_0:\fH^\R \to \fH^\R$, $J_0 f=i f$, be its complex  structure.  Let $\GL(\fH^\R)$ be the group of real linear invertible operators on $\fH^\R$.  The \textit{restricted orthogonal group} is defined by
$$
\O_{\mathrm{res}}(\fH^\R):=\{ \, O \in \GL(\fH^\R) \, : \,  O^\tau=O^{-1}, \,  [O,J_0] \in \cB_2(\fH^\R)         \,   \}.
$$
Here we write $O^\tau$ for the transpose of $O$ relative to the inner product $\Re  \PI{\, }{\, }$ of $\fH^\R$. Several properties of this group were studied in \cite{CO, N, Ott, PS, V}. In particular, $\O_{\mathrm{res}}(\fH^\R)$ is a Lie group having two connected components. Its topology is defined by the norm $\| O\|_{\mathrm{res}}:=\max\{ \|O\|, \,  \|[J_0, O]\|_2 \}$,  $O \in \O_{\mathrm{res}}(\fH^\R)$, and its Lie algebra is given by
$$
\mathfrak{o}_{\mathrm{res}}(\fH^\R):=\{ \, A \in \cB(\fH^\R) \, : \,  A^\tau=-A, \,  [A,J_0] \in \cB_2(\fH^\R)         \,   \}.
$$

Given an operator $A \in \cB(\fH^\R)$, we may write $A=A_c+ A_a$, where $A_a:=\frac{1}{2}(A+J_0AJ_0)$ and $A_c:=\frac{1}{2}(A-J_0AJ_0)$ are called the \textit{antilinear part} (or \textit{conjugate-linear part}) and \textit{linear part} of $A$, respectively. Observe that the maps $A\mapsto A_a$ and $A\mapsto A_c$ are projections, $A_aJ_0=-J_0A_a$ and $A_cJ_0=J_0A_c$. For $A \in \cB(\fH^\R)$ such that $A^\tau=-A$, note that $A \in \mathfrak{o}_{\mathrm{res}}(\fH^\R)$ if and only $A_a \in \cB_2(\fH^\R)$.
\end{rem}

Let $\mathfrak{u}(\fH \oplus \fH)$ denote the Lie algebra of $\U(\fH \oplus \fH)$ consisting of all skew-adjoint operators. 
In the next proposition we collect several facts on the group $\UB$, including its relation with the above restricted general linear and orthogonal groups.   We  abbreviate  $\rG:=\GL_2(\fH \oplus \fH,I,D)$ and $\fg:=\mathfrak{gl}_2(\fH \oplus \fH,I,D)$.







\begin{prop}\label{elementary properties}
$\UB$ is a real Lie  subgroup of $\rG$, 
whose Lie algebra is given by
\begin{align*}
\mathfrak{u}_{\mathrm{Bog}} & := \{ \, X \in \mathfrak{u}(\fH \oplus \fH)  :     X= I X I  ,    \, [X,D] \in \cB_2(\fH \oplus \fH)    \}   \\
 & = \left\{ \, \begin{pmatrix} x_1  &  x_2  \\   \bar{x}_2  &  \bar{x}_1   \end{pmatrix} \in \cB(\fH\oplus \fH) :  x_1=-x_1^*, \, x_2=-x_2^T \in  \cB_2(\fH)  \right\}=\fg_\R. 
\end{align*}
In particular, the  topology of $\UB$ is defined by the norm 
\begin{equation}\label{res norm ubog}
\left\|  \begin{pmatrix}  u  &  v   \\  \bar{v}  &  \bar{u}  \end{pmatrix} \right\|_{\res}= 2 \max\{ \|u\| , \, \|v\|_2 \}. 
\end{equation}
Furthermore, the following assertions hold: 
\begin{enumerate} 
\item[i)] If $U=\begin{pmatrix}  u  &  v   \\  \bar{v}  &  \bar{u}  \end{pmatrix} \in \UB$, then $u$ and $\bar{u}$ are Fredholm operators of index zero. 
\item[ii)] The map $\Xi: \O_{\mathrm{res}}(\fH^\R) \to \UB$ defined by
\begin{equation}\label{iso matrix op}
\Xi (O)=   \begin{pmatrix}  u  & v \\ \bar{v}  &  \bar{u} \end{pmatrix}   \text{ if and only if } O=u + v I_0 ,  \, u=O_c  , \,  v I_0=O_a ,
\end{equation}
is an isomorphism of Lie groups. 
\item[iii)] Let   $i \fg_\R:=\{ X \in \fg : X=X^*  \}$ be the self-adjoint part of $\fg$. The polar decomposition given by 
$$
\UB \times i \fg_\R \to \rG, \, \, \, \, (U,X) \mapsto Ue^X,
$$
is a diffeomorphism and the inclusion $\UB \hookrightarrow \rG$ is a homotopy equivalence. 
\item[iv)]  $\pi_0(\UB)=\mathbb{Z}_2$. The higher homotopy groups of $\UB$ are 8-periodic, and consequently, they are determined by $\pi_i(\UB^0)=\mathbf{0}$, $i=1,3,4,5$, $\pi_i(\UB^0)=\mathbb{Z}$, $i=2,6$, and $\pi_i(\UB^0)=\mathbb{Z}_2$, $i=7,8$. 
\item[v)] There is a $\mathbb{Z}_2$-valued index map defined by 
$$
\ind_{\UB}: \UB \to \mathbb{Z}_2, \, \, \, \,  \ind_{\UB} \left( \begin{pmatrix}  u   &   v  \\  \bar{v}  &   \bar{u} \end{pmatrix}\right)=\dim  \mathrm{ker}(u) \, \, (\text{mod  2}). 
$$
This  is a continuous morphism of groups, which parametrizes the two connected components of $\UB$.
\end{enumerate}
\end{prop}
\begin{proof}
First, it is easy to see that 
$\UB$ is closed in $\rG$ and $\ub=\{ X \in \fg : e^{tX} \in \UB, \, \forall t \in \R  \}$. 
Also note that $\ub=\fg_\R:=\{ X \in \fg : X^*=-X \}$ is the real form of $\fg$, which has a closed supplement in $\fg$ given by $i \fg_\R:=\{ X \in \fg : X^*=X \}$. According to  Theorem \ref{quotient Lie manifolds struct}, it remains to be shown  that
there are open sets $0 \in \cW \subseteq \fg$, $\d1=\d1_{\fH \oplus \fH} \in \cV \subseteq \rG$, such that the exponential map $\exp:\cW \to \cV$ is a diffeomorphism satisfying $\exp(\cW \cap \ub)=\cV \cap \UB$. The  non trivial inclusion is `$\supseteq$'. Pick $U=e^X \in \UB$, for some $X \in \cW$.  Eventually shrinking $\cV$ so that $\|U - \d1 \|_{\mathrm{res}}<1$, and then using  analytic functional calculus in the Banach algebra $\cB_2(\fH \oplus \fH, D)$, one gets that  $X=\log(U)=\sum_{n \geq 1} (-1)^{n+1} \frac{(U- \d1)^n}{n} \in \cB_2(\fH \oplus \fH, D)$. By this expression in terms of a series, it follows that $X^*=-X$ and $X=IXI$. Therefore, $X \in \ub$. Hence $\UB$ is a Lie subgroup of $\rG$, whose Lie algebra is given by $\ub$.


\medskip

\noi $i)$ Using the expressions in  \eqref{Bog 1}, and noting that $v \in \cB_2(\fH)$ and thus a compact operator,  it follows that $u$ is invertible in the Calkin algebra. Hence $u$ and $\bar{u}$ are Fredholm operators. We write $\ind(u)=\dim \ker(u) - \dim \ker(u^*)$ for its Fredholm index. But $U$ is unitary, so that it is a Fredholm operator of  index  zero on $\fH\oplus \fH$. It is a well-known property of the index of block operator matrices that one can calculate  $0=\ind(U)=\ind(u)+ \ind(\bar{u}).$ Since $\ind(u)=\ind(\bar{u})$, it follows that $\ind(u)=\ind(\bar{u})=0$.

\medskip

\noi $ii)$ We follow the construction in \cite[Section IV.2]{N} (see also \cite{PR}).  We begin by taking the complexification $\fH_\C:=(\fH^\R)_\C$, and denote by $J$ the extension  of $J_0$ to $\fH_\C$. One possible way to introduce the complexification is to put $\fH_\C=\fH^\R \times \fH^\R$ as a set equipped with the usual sum,  multiplication by complex scalars and  inner product. The extension of an arbitrary  linear operator $A$ on $\fH^\R$ to $\fH_\C$ is given by $A_\C(f,g)=(Af,Ag)$, for $f,g \in \fH^\R$.  Notice that we can decompose the space as $\fH_\C=\fH_\C^+ \oplus \fH_\C^-$, where $\fH_\C^{\pm}$ are the $\pm i$-eigenspaces of $J$.
Each vector in $\fH_\C^{\pm}$ can be written as $(f ,\mp i f)$. It can be checked that the operator  $U:\fH \to \fH_\C^+ $, $Uf=\frac{1}{\sqrt{2}}(f,-if)$ is a complex linear surjective isometry. 
Next we define the    conjugation $C:\fH_\C \to \fH_\C$, $C((f,g))=(f,-g)$. Clearly, we have that $\fH^\R\cong \fH^\R \times \{ 0 \}=\{  (f,g) \in \fH_\C \, : \,  C((f,g))=(f,g)   \}$. Also note that $C \fH_\C^\pm=\fH_\C^\mp$. 
There is an isometric isomorphism between $\fH\oplus \fH$ and $\fH_\C=\fH_\C^+ \oplus \fH_\C^-$ given by
$$
T: \fH\oplus \fH   \to   \fH_\C  , \, \, \, \, T((f,g))=Uf + CUI_0 g=\frac{1}{\sqrt{2}}\left(f + \bar{g}, i(\bar{g} - f)\right).  
$$
By using  this isomorphism we can identify operators on $\fH\oplus \fH$ with operators on $\fH_\C$, i.e.
$\cB(\fH\oplus \fH) \to \cB(\fH_\C), \, \, \, \, X \mapsto \Ad_T X=TXT^{-1}.$
Let $\U(\fH_\C) $ be the unitary group of $\fH_\C$. The extension of orthogonal operators to $\fH_\C$ gives the following isomorphism between the Lie groups
$$
\O_{\mathrm{res}}(\fH^\R) \to \{ \, U \in \U(\fH_\C)  \,  :  \, UC=CU ,  \,  [U,J] \in \cB_2(\fH_\C)   \, \}, \, \, \, \, \, 
O \mapsto O_\C . 
$$
 Therefore, there is a Lie group isomorphism given by 
\begin{equation}\label{isom Banach alg}
\Xi:  \O_{\mathrm{res}}(\fH^\R) \to   \UB, \, \, \, \, \, \,   \Xi(O) = T^{-1} O_\C  T.
\end{equation}
We only observe that the map $\Xi$ actually  takes values in $\UB$, which can be deduced from the characterization  given in \eqref{ubog and g}, and noting that  $iD=\Ad_{T^{-1}}J$ and  $I=\Ad_{T^{-1}}C$.  All the properties of Lie group isomorphism are  straightforward to check. Finally, we remark that this isomorphism can be written in terms of the block matrix operators, and the linear and antilinear parts as in Eq. \eqref{iso matrix op}.

\medskip

\noi $iii)$ These are proved in  \cite[Prop. III.2]{N} and \cite[Prop. III.8]{N}.

\medskip

\noi $iv)$  This is \cite[Prop. III.14]{N} combined with the homotopy equivalence of  item $iii)$.

\medskip 

\noi $v)$ For instance, we refer to \cite{A87} for the construction of this $\mathbb{Z}_2$-valued index map. Alternatively, an index was given in \cite{CO} for a general restricted orthogonal group on a real Hilbert space with a complex structure. For our particular real Hilbert space $\fH^\R$ with complex structure $J_0$ the mentioned index is given by 
$
\ind_{\O_{\mathrm{res}}}: \O_{\mathrm{res}}(\fH^\R ) \to \mathbb{Z}_2$, $\ind_{\O_{\mathrm{res}}}(O)=\dim_\C \ker(O-J_0 O J_0) \, \, (\text{mod  2}),
$
where the subscript $\C$ means that the complex dimension must be computed. 
One can verify  that $\ind_{\UB}(U)=\ind_{\O_{\mathrm{res}}}(\Xi^{-1}(U))$.
\end{proof}

\subsection{A reductive structure on orbits of g1-pdms and quasi-free states}

In this subsection, we study the orbits of the conjugacy action of $\UB$ on admissible g1-pdms  as  reductive homogeneous spaces of $\UB$, or equivalently, orbits of  quasi-free states with finite particle number.

\begin{rem}\label{proj ubog diag}
We first point out some useful spectral properties of g1-pdms. 

\medskip

\noi $i)$ According to Theorem \ref{ubog diagonalization}, any $\Gamma \in \DF$ can be diagonalized 
\begin{equation*}
W\Gamma W^*=\begin{pmatrix}  \Lambda   &   0  \\   0   &  \d1  - \Lambda \end{pmatrix},
\end{equation*}
for some unitary $W \in \UB$. In what follows, it will be convenient to express $\Lambda$ using the spectral theorem for trace-class operators. Indeed, there is a family of  orthogonal projections $\{ p_i \}_{i =0}^r$ ($0 \leq r \leq \infty$) satisfying $\sum_{i=0}^r p_i=\d1$ (strong operator topology convergence if $r=\infty$), $p_ip_j=\delta_{ij}p_i$, $\dim(\ran(p_i)):=m_i < \infty$ for all $i \geq 1$ and 
$$
\Lambda=\sum_{i=1}^r \lambda_i p_i .
$$ 
Here the convergence is understood in the $\| \, \cdot \, \|_1$-norm if $r=\infty$; and the eigenvalues satisfy $\lambda_i \in (0,\frac{1}{2}]$, $i \geq 1$, $\lambda_i \neq \lambda _j$ if $i \neq j$ and $\sum_{i=1}^r m_i \lambda_i  < \infty$. Observe that  one can also assume  $p_i=\bar{p}_i=p_i^T$ because $\Lambda$ is diagonal with respect to the fixed orthonormal basis $\{ \varphi_k\}_{k\geq 1}$.  Finally, we put $\lambda_0=0$, so that $p_0$ is the projection onto $\ker(\Lambda)$, which satisfies $\dim(\ran(p_0)):=m_0 \in [0,\infty]$.
 
\medskip

\noi $ii)$ Notice that the spectrum of $\Gamma$ is then given by 
$$
\sigma(\Gamma)=\{ 0,1 \} \cup \{ \lambda_i \}_{i=1}^r \cup \{ 1- \lambda_i\}_{i=1}^r .
$$ 
Since $\Gamma$ is a self-adjoint operator, it is well known that $\sigma(\Gamma)=\sigma_p(\Gamma) \cup \sigma_{ess}(\Gamma)$, where $\sigma_p(\Gamma)$ and $\sigma_{ess}(\Gamma)$ indicate the point spectrum  (consisting of all eigenvalues of finite multiplicity) and the essential spectrum of $\Gamma$, respectively. Clearly, we have $\sigma_p(\Gamma)=\{ \lambda_i \}_{i=1}^r \cup \{ 1- \lambda_i\}_{i=1}^r$ and $\sigma_{ess}(\Gamma)=\{ 0,\, 1\}$. 
\end{rem}

\begin{rem} 
Notice that $\UB$ acts on $\DF$ by unitary conjugation:  
$$
U \cdot \Gamma=U\Gamma U^*,  \, \, \, \, \, \, U \in \UB , \, \, \, \Gamma \in \DF.
$$ 
It is  well-known in the Hartree-Fock-Bogoliubov literature that this is actually an action (see, e.g.,  \cite{BLS94}). We add the matrix form of $U \cdot \Gamma$ to  briefly show this fact here.
If
$\Gamma=\Gamma[\gamma, \alpha]$,  $U=\begin{pmatrix}  u  &   v   \\   \bar{v} &  \bar{u}  \end{pmatrix}$,
then
\begin{equation}\label{action ubog}
U \cdot \Gamma=
\begin{pmatrix}
u \gamma u^* + v \alpha^*u^* + u \alpha v^* + v(\d1 - \bar{\gamma})v^* &
u\gamma \bar{v}^*  + v \alpha^* \bar{v}^* + u \alpha \bar{u}^* + v(\d1-\bar{\gamma})\bar{u}^* \\
\bar{v}\gamma u^* + \bar{u} \alpha^* u^* + \bar{v} \alpha v^* + \bar{u} (\d1 - \bar{\gamma})v^* & 
\bar{v} \gamma \bar{v}^* + \bar{u} \alpha^* \bar{v}^* + \bar{v} \alpha \bar{u}^* + \bar{u} (\d1 - \bar{\gamma}) \bar{u}^*
\end{pmatrix}.
\end{equation}
From this expression,   we deduce that the entry $(U\cdot \Gamma)_{11} \in \cB_1(\fH)$ because $\alpha, v \in \cB_2(\fH)$
(see Remark \ref{d trace rem} $ii)$). Also it easily follows that $(U\cdot \Gamma)_{22}=\d1 - \overline{(U\cdot \Gamma)}_{11}$ using that 
$\alpha^T=-\alpha$, or equivalently, $\alpha^*=-\bar{\alpha}$. The remaining conditions to ensure that $U \cdot \Gamma \in \DF$ are trivial. 
\end{rem}

Given $\Gamma \in \DF$, we consider its orbit 
$$
\cO(\Gamma) =\{ U\Gamma U^* : U \in \UB  \}.
$$
In order to prove that these orbits are smooth homogeneous spaces of $\UB$, we have to consider the isotropy group of $\UB$ at $\Gamma$, that is,
$$
\UB^\Gamma :=\{ U \in \UB : U \Gamma U^* =\Gamma \},
$$
and prove the  following:

\begin{lem}\label{lemma Lie subg g1-pdms}
$\UB^\Gamma$ is a Lie subgroup of $\UB$.   
\end{lem}
\begin{proof}
First, it is clear $\UB^\Gamma$ is a closed subgroup of $\UB$ in topology defined by the norm in  \eqref{res norm ubog}. Set
$$
\ubg:=\{ X \in \mathfrak{u}_{\mathrm{Bog}}  :  e^{tX} \in \UB^\Gamma, \, \forall t \in \R \},
$$
which is then a closed Lie subalgebra of $\ub$ by Theorem \ref{quotient Lie manifolds struct}. By the same result we have to check that  there are open sets $0 \in \cW \subseteq \mathfrak{u}_{\mathrm{Bog}}$, $1 \in \cV \subseteq \UB$, such that the exponential map $\exp:\cW \to \cV$ is a diffeomorphism satisfying $\exp(\cW \cap \ubg)=\cV \cap \UB^\Gamma$, and the subspace $\ubg$ has a closed supplement in $\ub$.


  It will be useful to observe that 
\begin{equation}\label{ubg as a commutator}
\ubg=\left\{  X \in \mathfrak{u}_{\mathrm{Bog}}   :    X\Gamma = \Gamma X \right\}.
\end{equation}
Indeed, one inclusion can be shown by taking the derivative of the curve $\gamma(t)=e^{tX}$ at $t=0$; meanwhile the other inclusion follows from the fact that $X\Gamma=\Gamma X$ implies that $(tX)^n\Gamma=\Gamma (tX)^n$, $n\geq 0$, and consequently, $e^{tX}\Gamma=\Gamma e^{tX}$. 
 In order to check $\exp(\cW \cap \ubg)=\cV \cap \UB^\Gamma$, we take $U=e^X \in \UB^\Gamma$, where $X \in \cW$. Using analytic functional calculus in the Banach algebra $\cB_2(\fH \oplus \fH, D)$, we get that $X=\log(U)=\sum_{n \geq 1} (-1)^{n+1} \frac{(U- \d1)^n}{n}$, whenever $\cV$ is a sufficiently small identity neighborhood. Since $U\Gamma = \Gamma U$ yields $(U-\d1)^n\Gamma=\Gamma (U-\d1)^n$, $n \geq 1$, we obtain that $\log(U)\Gamma=\Gamma \log(U)$, which proves the desired inclusion.  The other inclusion is trivial.

Now we show that $\ubg$ admits a closed supplement in $\ub$. Note that   $W \UB^\Gamma W^*=\UB^{W\Gamma W^*}$ for every $W \in \UB$. Then it follows that $\UB^\Gamma$ is a Lie subgroup if and only if $\UB^{W\Gamma W^*}$ is. By Theorem \ref{ubog diagonalization} we may therefore assume that $\Gamma$ has  the form  
$\Gamma=\Gamma[\Lambda,0]$, 
where $\Lambda$  is a trace-class diagonal operator with respect to the fixed orthonormal basis $\{ \varphi_k\}_{k\geq 1}$. We now use the notation and properties given in Remark \ref{proj ubog diag}, where $\Lambda$ is expressed as $\Lambda=\sum_{i=1}^r \lambda_i p_i$ for a decomposition of the identity  $\{ p_i\}_{i=0}^r$ ($0\leq r \leq \infty$).  Recall that $\Gamma$ has spectrum given by  $\sigma(\Gamma)=\{ 0,1 \} \cup \{ \lambda_i \}_{i=1}^r \cup \{ 1- \lambda_i\}_{i=1}^r$, where $\lambda_i \in (0,\frac{1}{2}]$ are distinct eigenvalues of finite multiplicity. 


Then for any operator 
$
X= \begin{pmatrix}
x_1 & x_2 \\
\bar{x}_2 & \bar{x}_1
\end{pmatrix} \in \ub,
$
the condition $X\Gamma=\Gamma X$ is equivalent to have 
\begin{equation}\label{2nd charac}
x_1 \Lambda = \Lambda x_1 , \, \, \, \, \,  x_2(\d1-\Lambda)=\Lambda x_2.
\end{equation}
From  $x_1 \Lambda =\Lambda x_1$, it follows that $\lambda_j p_i x_1p_j=\lambda_i p_i x_1 p_j$ for $i,j \geq 0$. Thus, $(\lambda_i-\lambda_j)p_i x_1 p_j=0$, which yields $p_ix_1p_j=0$, whenever $i\neq j$, $i,j \geq 0$.  Hence $x_1$ must have the block diagonal form $x_1=\sum_{i=0}^r p_i x_1 p_i$.

  On the other hand, $x_2(\d1-\Lambda)=\Lambda x_2$ implies that $(1-\lambda_j - \lambda_i) p_ix_2p_j=0$, $i,j \geq 0$. But $1-\lambda_j - \lambda_i=0$ only if $\lambda_i=\lambda_j=\frac{1}{2}$. Thus, we have two cases. In the first case, $x_2=0$ when $\lambda=\frac{1}{2}$ is not an eigenvalue of $\Gamma$.  In the second case, we may w.l.o.g. assume that $\lambda_1=\frac{1}{2}$ is an eigenvalue of $\Gamma$ and $p_1$ is the finite-rank projection onto the corresponding eigenspace. Therefore $p_i x_2 p_j=0$ for all  $i,j \geq 0$, except  when $i=j=1$, which is equivalent to say that $x_2=p_1 x_2 p_1$.  Hence when $\frac{1}{2} \notin \sigma(\Gamma)$ the Lie algebra of the isotropy group can be written as 
\begin{equation*}
\ubg=\left\{  \begin{pmatrix}  x_1  &  0   \\ 0 & \bar{x}_1   \end{pmatrix} \in \ub: x_1=\sum_{i=0}^r p_i x_1  p_i          \right\},
\end{equation*}
meanwhile when $\frac{1}{2} \in \sigma(\Gamma)$ is 
\begin{equation*}
\ubg=\left\{  \begin{pmatrix}  x_1  &  x_2   \\ \bar{x}_2 & \bar{x}_1   \end{pmatrix} \in \ub: x_1=\sum_{i=0}^r p_i x_1  p_i, \, x_2=p_1 x_2 p_1          \right\}. 
\end{equation*}
Then a continuous projection $\cE_\Gamma: \ub \to \ub$ with range $\ran(\cE_\Gamma)=\mathfrak{u}_{\mathrm{Bog}}^\Gamma$ is given by 
\begin{equation*}
\cE_\Gamma\left(   \begin{pmatrix} x_1 & x_2 \\  \bar{x}_2 &   \bar{x}_1 \end{pmatrix}   \right)  =  
\begin{cases}
      \begin{pmatrix} \sum_{i=0}^r p_i x_1 p_i   &   0  \\
0  &  \sum_{i=0}^r p_i \bar{x}_1 p_i 
\end{pmatrix} &   \text{ if } \frac{1}{2}  \notin \sigma(\Gamma),\\
      \begin{pmatrix} \sum_{i=0}^r p_i x_1 p_i   &   p_1 x_2 p_1  \\
p_1 \bar{x}_2 p_1  &  \sum_{i=0}^r p_i \bar{x}_1 p_i 
\end{pmatrix} &  \text{ if } \frac{1}{2}  \in \sigma(\Gamma).
					\end{cases}       
\end{equation*}
The convergence of the series is in the strong operator topology when $r=\infty$.    
This proves that $\mathfrak{u}_{\mathrm{Bog}}^\Gamma$ has the closed supplement $\mathfrak{m}_\Gamma:=\ker(\cE_\Gamma)$ in both cases.
\end{proof}

\begin{rem} Repeating the same computations at the group level that we have done at the Lie algebras level in the previous proof, we find that when $\frac{1}{2} \notin \sigma(\Gamma)$ the isotropy group is given by 
\begin{equation}\label{isot group sin medio}
\UB^\Gamma =\left\{  \begin{pmatrix}   u &      0   \\  0  &  \bar{u}  \end{pmatrix} \in \UB : u=\sum_{i=0}^r  p_i u p_i   \right\}.
\end{equation}
In the case where  $\frac{1}{2} \in \sigma(\Gamma)$, the isotropy group is given by
\begin{equation}\label{isot group sec case}
\UB^\Gamma =\left\{  \begin{pmatrix}   u &      v   \\  \bar{v}  &  \bar{u}  \end{pmatrix} \in \UB : u=\sum_{i=0}^r  p_i u p_i, \, v=p_1vp_1   \right\}.
\end{equation}
\end{rem}

It  wil be useful to consider the projections $\cE_\Gamma$ defined in the proof of Lemma \ref{lemma Lie subg g1-pdms}.  The  following  definition is motivated by the conditional expectations associated to normal diagonalizable operators given in \cite{BL23}.

\begin{fed}\label{cond exp}
Let $\Gamma=\Gamma[\Lambda, 0]$ be a diagonal g1-pdm  which defines a family of projections $\{  p_i\}_{i=0}^r$ ($0 \leq  r \leq \infty$) satisfying the properties of Remark \ref{proj ubog diag}. The \textit{conditional expectation associated to $\Gamma$} is the map 
$\cE_\Gamma : \ub \to \ub$ defined by 
\begin{equation}\label{Cond expectation 0}
\cE_\Gamma\left(   \begin{pmatrix} x_1 & x_2 \\  \bar{x}_2 &   \bar{x}_1 \end{pmatrix}   \right)  =  
\begin{cases}
      \begin{pmatrix} \sum_{i=0}^r p_i x_1 p_i   &   0  \\
0  &  \sum_{i=0}^r p_i \bar{x}_1 p_i 
\end{pmatrix} &   \text{ if } \frac{1}{2}  \notin \sigma(\Gamma),\\
      \begin{pmatrix} \sum_{i=0}^r p_i x_1 p_i   &   p_1 x_2 p_1  \\
p_1 \bar{x}_2 p_1  &  \sum_{i=0}^r p_i \bar{x}_1 p_i 
\end{pmatrix} &  \text{ if } \frac{1}{2}  \in \sigma(\Gamma).
					\end{cases}       
\end{equation}
The convergence is in the strong operator topology when $r=\infty$.
\end{fed}

\begin{rem} It is straightforward to check  that the  conditional expectations in Eq. \eqref{Cond expectation 0} satisfy properties similar to that of conditional expectations in operator algebras: $\cE_\Gamma^2=\cE_\Gamma$, $\ran(\cE_\Gamma)=\ubg$, $\|\cE_\Gamma\|=\sup_{\| X\|_{\mathrm{res}}=1}\| \cE_\Gamma(X)\|_{\mathrm{res}}=1$, and $\cE_\Gamma(UZU^*)=U\cE_\Gamma(Z)U^*$, whenever $U \in \UB^\Gamma$ and $Z \in \ub$.
\end{rem}

\begin{teo}\label{g1pdm smooth homog spaces}
Let $\Gamma \in \DF$. Then the following assertions hold. 
\begin{enumerate}
\item[i)] The orbit
$
\cO(\Gamma)
 \cong \UB / \UB^\Gamma$ is  a  reductive homogeneous space of   $\UB$. Its tangent space at $\Gamma_1$ can be identified with
the real Banach space
\begin{equation*}
T_{\Gamma_1}\cO(\Gamma)= \{ [X,\Gamma_1 ] : X \in \ub  \},
\end{equation*}
equipped with the norm $\| [X  ,  \Gamma_1 ]  \|_{\Gamma_1}:=\inf \{  \|  X+ Y \|_{\mathrm{res}} : Y\Gamma_1=\Gamma_1 Y  \}$.
\item[ii)] Let us denote by 
$$
\UB^+:=\{  U \in \UB :  \ind_{\UB}(U)=0 \}, \, \, \, \, \, \UB^-:=\{  U \in \UB :  \ind_{\UB}(U)=1 \}.
$$
If $\frac{1}{2} \notin \sigma(\Gamma)$, then $\cO(\Gamma)=\UB^+ \cdot \Gamma \cup \UB^- \cdot \Gamma$,  is  the union of two connected components. If $\frac{1}{2} \in \sigma(\Gamma)$, then $\cO(\Gamma)=\UB^+ \cdot \Gamma=\UB^- \cdot \Gamma$ is connected. 
\end{enumerate}
\end{teo}
\begin{proof}
$i)$ The map $\cO(\Gamma) \to \UB / \UB^\Gamma$, $W\Gamma W^* \mapsto W \UB^\Gamma$, is a bijection. Thus, we endow $\cO(\Gamma)$ with the manifold structure making such map into a diffeomorphism. The manifold structure of $\UB / \UB^\Gamma$ follows by using  Theorem \ref{quotient Lie manifolds struct} and Lemma \ref{lemma Lie subg g1-pdms}. Indeed, by these results $\UB / \UB^\Gamma$ becomes  a smooth homogeneous space of $\UB$. 

The assertion about tangent spaces now follows by using that the map $\pi: \UB \to \cO(\Gamma)$, $\pi(U)=U\Gamma U^*$, is a submersion. Indeed, its tangent map at $U \in \UB$, $\U \cdot \Gamma=\Gamma_1$, satisfies $T_{\Gamma_1}\cO(\Gamma)=\ran(T_U \pi)\cong \ub / \mathfrak{u}_{\mathrm{Bog}}^{\Gamma_1}\cong \{ [X ,\Gamma_1] : X \in \ub  \}$, where $[X,\Gamma_1]=X\Gamma_1 - \Gamma_1 X$ and $\ker(T_U \pi)=\mathfrak{u}_{\mathrm{Bog}}^{\Gamma_1}$ is the Lie algebra of the isotropy group at $\Gamma_1$. The norm given in the statement is an expression for the quotient norm in $\ub / \mathfrak{u}_{\mathrm{Bog}}^{\Gamma_1}$.

To prove that $\cO(\Gamma)$ is a  reductive homogeneous space we may assume that $\Gamma$ is diagonal, take $\cE_\Gamma$ its conditional expectation  and $\mathfrak{m}_\Gamma=\ker(\cE_\Gamma)$. From the proof of Lemma \ref{lemma Lie subg g1-pdms}, we know that $\ubg \oplus \mathfrak{m}_\Gamma=\ub$. Since $\cE_\Gamma(UXU^*)=U\cE_\Gamma(X)U^*$, for $U \in \UB^\Gamma$ and $X \in \ub$, it follows that $\Ad_U (\mathfrak{m}_\Gamma)=\mathfrak{m}_\Gamma$.  Then, for $\Gamma_1=U\Gamma U^*$, set $\mathfrak{m}_{\Gamma_1}:=U
\mathfrak{m}_\Gamma  U^*$. This is well defined since $U \Gamma U^*=W \Gamma W^*$ implies $W^* U \in \UB^\Gamma$. Thus, $\cE_\Gamma=\Ad_{W^*U} \circ \cE_\Gamma \circ \Ad_{U^*W}$, which yields 
$U \mathfrak{m}_\Gamma U^*=W \mathfrak{m}_\Gamma W^*$. 
By similar properties the distribution $\{ \mathfrak{m}_{\Gamma_1} \}_{\Gamma_1 \in \cO(\Gamma)}$ clearly satisfies $\ub=\mathfrak{m}_{\Gamma_1}\oplus \mathfrak{u}_{\mathrm{Bog}}^{\Gamma_1}$ and  $\Ad_U (\mathfrak{m}_{\Gamma_1})=\mathfrak{m}_{\Gamma_1}$, for $U \in \UB^{\Gamma_1}$. Fix $\Gamma_1 \in \cO(\Gamma)$. For any $U \in \UB$ such that $\Gamma_1=U\Gamma U^*$, the projection $\cE_{\Gamma_1}:=U \cE_\Gamma (U^* \, \cdot \, U)U^*$ 
is the unique projection satisfying $\ran(\cE_{\Gamma_1})=\mathfrak{u}_{\mathrm{Bog}}^{\Gamma_1}$ and $\ker(\cE_{\Gamma_1})=\mathfrak{m}_{\Gamma_1}$. On the other hand, the map $\pi_{\Gamma_1}: \UB \to \cO(\Gamma)$, $\pi_{\Gamma_1}(U)=U\Gamma_1  U^*$, is a  submersion. 
This is equivalent to the existence of a smooth local cross section at every point (see \cite[Cor. 8.3]{Up85}). In particular,  there are an open set $\cV_{\Gamma_1} \subseteq \cO(\Gamma)$, $\Gamma_1 \in \cV_{\Gamma_1}$ and a smooth map $s_{\Gamma_1}:\cV_{\Gamma_1} \to \UB$ such that $\pi_{\Gamma_1} \circ s_{\Gamma_1} = \mathrm{id}|_{\cV_{\Gamma_1}}$. Consider the smooth map $F: \UB \to \cB(\ub)$, $F(U)=\d1_{\cB(\ub)} - U \cE_\Gamma (U^* \, \cdot \, U)U^*$ and the map $f: \cO(\Gamma_1) \to \cB(\ub)$, $f(\Gamma')=\d1_{\cB(\ub)} -\cE_{\Gamma'}$. Therefore, $f=F \circ s_{\Gamma_1}$ locally at each $\Gamma_1$. Since $\Gamma_1$ is arbitrary, we get that $f$ is a smooth map, and consequently, the distribution  $\{ \mathfrak{m}_{\Gamma_1} \}_{\Gamma_1 \in \cO(\Gamma)}$ is smooth. Hence $\cO(\Gamma)$ is a  reductive homogeneous space.

\medskip

\noi $ii)$ According to Proposition \ref{elementary properties}, the $\mathbb{Z}_2$-valued index on $\UB$ parametrizes the two connected components of this group $\UB^+$ and $\UB^-$. Therefore, $\UB^+ \cdot \Gamma=\{  U\Gamma  U^* : U \in \UB^+ \}$ is a connected set because the map $\pi_{\Gamma}: \UB \to \cO(\Gamma)$, $\pi_{\Gamma}(U)=U\Gamma  U^*$ is continuous. Let us show that $\UB^+ \cdot \Gamma$ is a connected component of $\cO(\Gamma)$. Recall that path and connected components coincide in manifolds. So take a continuous path $\delta:[0,1] \to \cO(\Gamma)$, $\delta(0)=\Gamma$ and $\delta(1)=\Gamma_1$. We have to prove that $\Gamma_1 \in \UB^+ \cdot \Gamma$. 

By item $i)$ the maps $\pi_{\Gamma'}: \UB \to \cO(\Gamma)$, $\pi_{\Gamma'}(U)=U\Gamma'  U^*$, are submersions. 
As we have observed above,  there are an open set $\cV_{\Gamma'} \subseteq \cO(\Gamma)$, $\Gamma' \in \cV_{\Gamma'}$ and a smooth map $s_{\Gamma'}:\cV_{\Gamma'} \to \UB$ such that $\pi_{\Gamma'} \circ s_{\Gamma'} = \mathrm{id}|_{\cV_{\Gamma'}}$. 
By a standard compactness argument, there are a partition $0=t_0<t_1<\ldots <t_n=1$, open connected sets $\cV_i:=\cV_{\delta(t_i)}$, and smooth sections $s_i:=s_{\delta(t_i)}$ of $\pi_{\delta(t_i)}$, $i=0, \ldots, n$,  such that $\delta([0,1]) \subseteq \bigcup_{i=0}^n \cV_i$. 
Set $\delta(t_i):=\Gamma^{(i)}$, $i=0, \ldots, n$, where $\Gamma=\Gamma^{(0)}$ and $\Gamma_1=\Gamma^{(n)}$.
For $i=0$, we can have $\ind_{\UB}(s_0(\Gamma))= 0$ or $\ind_{\UB}(s_0(\Gamma))=1$. In the first case, we write $\tilde{s}_0=s_0$. In the second case,  we modify the section by defining $\tilde{s}_0(\Gamma')=s_0(\Gamma')s_0(\Gamma)$. Now this  a smooth section which satisfies
$\ind_{\UB}(\tilde{s}_0(\Gamma))=0$, and using that $\cV_0$ is connected, it also holds that $\ind_{\UB}(\tilde{s}_0(\Gamma^{(1)}))=0$. Notice that $\tilde{s}_0(\Gamma^{(1)})\Gamma \tilde{s}_0(\Gamma^{(1)})^*=\Gamma^{(1)}$. We may modify the other sections 
$s_i$ for $i=1, \ldots, n-1$ to obtain that $\ind_{\UB}(\tilde{s}_i(\Gamma^{(i+1)}))=0$ and
$$  
\tilde{s}_{n-1}(\Gamma^{(n)}) \ldots \tilde{s}_1(\Gamma^{(2)})\tilde{s}_0(\Gamma^{(1)})\Gamma \tilde{s}_0(\Gamma^{(1)})^*\tilde{s}_1(\Gamma^{(2)})^* \ldots \tilde{s}_{n-1}(\Gamma^{(n)})^*=\Gamma^{(n)}=\Gamma_1 \, .
$$
This means that $\Gamma_1 \in \UB^+ \cdot \Gamma$. Thus, $\UB^+ \cdot \Gamma$ is a connected component. 

Clearly, $\cO(\Gamma)=\UB^+ \cdot \Gamma \cup \UB^- \cdot \Gamma$. Fix $R \in \UB^-$.  Note that the map $\UB^+ \cdot \Gamma \to \UB^- \cdot \Gamma$, $U\Gamma U^* \to RU\Gamma U^*R^*$, is continuous. Therefore,
$\UB^- \cdot \Gamma$ is also connected. To show that it is a connected component, it suffices to prove that $\UB^+ \cdot \Gamma \cap \UB^- \cdot \Gamma=\emptyset$. Assume that $U\Gamma U^*=V\Gamma V^*$ for $U \in \UB^+$ and $V \in \UB^-$. Then we have $\ind_{\UB}(V^*U)=1$ and $V^*U \in \UB^\Gamma$. Note that $W\UB^{\pm}W^*=\UB^\pm$ for every $W \in \UB$. Thus, we may further assume that $\Gamma$ has a diagonal form. When $\frac{1}{2} \notin \sigma(\Gamma)$ the isotropy group is descirbed in \eqref{isot group sin medio}.
Thus we get that $P_+ (V^*U)|_{\fH \oplus \{ 0\}}$ is a block unitary operator. In particular, $\ind_{\UB}(V^*U)=0$, a contradiction. This proves that $\UB^- \cdot \Gamma$ is also a connected component. 

Finally, we suppose that $\frac{1}{2} \in \sigma(\Gamma)$. Again we can conjugate $\Gamma$ by a Bogoliubov transformation, so that $\Gamma$ is assumed to be diagonal. Now we use the expression of the isotropy group in \eqref{isot group sec case}.
Take a vector $\varphi_1 \in \ran(p_1)=\ker(\Lambda - \frac{1}{2} \d1)$ ($\lambda_1=\frac{1}{2}$). Define a Bogoliubov transformation $S$ by $S((\varphi_1,0))=(0,\varphi_1)$, $S((0,\varphi_1))=(\varphi_1,0)$, and that leaves fixed all the other vectors of the basis $\{ (\varphi_i,0) \}_{i\geq 1}\cup \{ (0,\varphi_j) \}_{j \geq 1}$.    Then $S\Gamma S^*=\Gamma$. But $\UB^+S=\UB^-$ because $\ind_{\UB}(S)=1$, so that $\UB^+ \cdot \Gamma= \UB^+S \cdot \Gamma  =\UB^- \cdot \Gamma=\cO(\Gamma)$.  
\end{proof}

\begin{cor}
$\DF$ is a smooth manifold. 
\end{cor}
\begin{proof}
Notice that $\DF$ is the disjoint  union of orbits, 
$
\DF=\bigsqcup_{[\Gamma] \in \DF / \approx} \cO(\Gamma),
$
where we write $\approx$ for the equivalence relation induced by the orbits. 
\end{proof}

\begin{rem}
We apply the previous results on g1-pdms to analyze the set $\Z$ of all quasi-free states. According to Theorem \ref{bog trans and unit impl}, for any $U \in \UB$, there exists a unitary implementer $\dU=\dU_U:\cF \to \cF$. The implementer is unique up to a constant in $\mathbb{T}$. Thus, given $\omega \in \Z$, the state defined by $\omega_U(A)=\omega(\dU A \dU^*)$,  $A \in \cA$, is independent of the unitary implementer.  Also as we have stated in Remark \ref{d trace rem} $i)$, it can be checked that $\omega_U \in \Z$. Furthermore, the map 
$$
\UB \times \Z \to \Z , \, \, \, \, \, (U, \omega) \mapsto U \cdot \omega =\omega_U 
$$
turns out to be an action of $\UB$ on $\Z$. This follows immediately by using  \eqref{fock space} to show that $\dU_{UV}=\lambda \dU_U \dU_V$, for $U,V \in \UB$ and some $\lambda \in \mathbb{T}$.
\end{rem}

\begin{cor}
The following assertions hold:
\begin{enumerate}
\item[i)]  For each $\omega \in \Z$, the orbit
$
\cO(\omega):=\{ \omega_U :   U \in \UB  \}
$
 is  a reductive homogeneous space of $\UB$. Consequently, $\Z$ is a smooth manifold. 
\item[ii)] The map $\Phi: \Z \to \cD$, $\Phi(\omega)=\Gamma_\omega$ is a smooth diffeomorphism.
\end{enumerate}
\end{cor}
\begin{proof}
$i)$ Using that the map $\Phi$ satisfies $\Gamma_{\omega_U}=\Phi(\omega_U)=U^*\Gamma_\omega U$, it follows that the isotropy group at $\omega$ is given by
$$
\{ U \in \UB : \omega_U = \omega \}=\{ U \in \UB :  U \Gamma_\omega = \Gamma_\omega U \}=\UB^{\Gamma_\omega}.
$$
Hence the isotropy group is a Lie subgroup of $\UB$ by Lemma \ref{lemma Lie subg g1-pdms},  and following the same proof as in Theorem \ref{g1pdm smooth homog spaces}, the orbit  $\cO(\omega)\cong \UB / \UB^{\Gamma_\omega}$ is a reductive homogeneous space of $\UB$.  
The assertion about the manifold structure of $\Z$ now follows by expressing this set as the union of disjoint orbits. 

\medskip

\noi $ii)$ The manifold structures of both $\Z$ and $\DF$ are obtained as the disjoint union of orbits. Therefore we may consider the restriction $\Phi:\cO(\omega) \to \cO(\Gamma_\omega)$ to prove our statement. The result is a consequence of being  $\cO(\omega)$ and $\cO(\Gamma_\omega)$ both diffeomorphic to $\UB/\UB^{\Gamma_\omega}$. Indeed, a chart at $\omega$ is given by $\phi^{-1}=\pi_\omega \circ \exp$, where $\pi_\omega:\UB \to \cO(\omega)$, $\pi_\omega(U)=\omega_U$, and $\exp$ is the exponential map of $\UB$ (see \cite[Lemma 4.21]{B06}). Moreover, $\phi$ is defined in an open set $\cW \subseteq \cO(\omega)$, $\omega \in \cW$, and $\phi(\cW)$ is an open set in $\mathfrak{m}_{\Gamma_\omega} \cong \ub / \mathfrak{u}_{\mathrm{Bog}}^{\Gamma_\omega}$, where $\mathfrak{m}_{\Gamma_\omega}$ is the closed subspace from the reductive structure. Similarly, a chart at $\Gamma_\omega$ is of the form $\psi^{-1}=\pi_{\Gamma_\omega} \circ \exp$, and it is a homeomorphism between open sets in $\cO(\Gamma_\omega)$ and $\mathfrak{m}_{\Gamma_\omega}$. Thus, the map $\Phi$ can be expressed locally as $(\psi \circ \Phi \circ \phi^{-1})(X)=(\psi \circ \Phi)(\omega_{e^X})=\psi(e^{-X}\Gamma e^X)=-X$. Hence $\Phi$ is a diffeomorphism. 
\end{proof}

\begin{rem}
We end this section with some remarks on a special type of orbits. It is well known that using the bijection between g1-pdms and quasi-free states the projection 
$P_-=\begin{pmatrix}   0  &   0  \\   0  &  \d1  \end{pmatrix}$ corresponds to the state $\omega_-$ defined 
by $\omega_-(A)=\PI{\Omega}{A\Omega}$, $A \in \cB(\cF)$, where $\Omega$ is the vacuum vector. This is the extension of the Fock state defined in the CAR algebra (see, e.g., \cite{A87}).

\medskip

\noi $i)$ The  orbit
$$
\cO(P_-)=\left\{  U \begin{pmatrix}   0  &   0  \\   0  &  \d1  \end{pmatrix} U^*  : U \in \UB \right\}
$$
is indeed isomorphic to an \textit{isotropic restricted Grassmannian}, which is a reductive homogeneous space of $\O_{\mathrm{res}}(\fH^\R)$ and  K\"ahler manifold that naturally shows up in the literature of loop groups (see \cite[Sec. 12.4]{PS}). Recalling the notation in the proof of Lemma \ref {elementary properties} $ii)$, we say that a (closed) subspace $W$ of $\fH_\C$ is maximal isotropic if $C(W)=W^\perp$, and we write $P_W$ for the orthogonal projection onto $W$.  Then, the isotropic restricted Grassmannian associated to $(\fH^\R,\fH_\C^-)$ is given by
$$
\cI_{\mathrm{res}}:=\{ W   : W \text{ is maximal isotropic}, \, P_W  -  P_{\fH_\C^-} \in \cB_2(\fH_\C) \}.
$$
The map $\cO(P_-) \to \cI_{\mathrm{res}}$, $UP_-U^* \mapsto \Ad_T(U)(\fH_\C^-)$, gives the identification between these homogeneous spaces.

\medskip

\noi $ii)$ The reductive structure induces a linear connection on every reductive homogeneous space (see \cite{KN93}, or in the infinite-dimensional case \cite{MR92}). For the orbit $\cO(P_-)$ the geodesics of this connection can be explicitly computed. For instance,  the geodesic starting at $P_-$ with velocity vector $[X,P_-]$, for $X=\begin{pmatrix}  0 & y \\  \bar{y} &  0  \end{pmatrix} \in \mathfrak{m}_{P_-}$, $y=-y^T \in \cB_2(\fH)$, is given by
$$\delta(t)=e^{tX} P_- e^{-tX} =\begin{pmatrix} \cos(|ty^*|) &  ty \mathrm{sinc}(|ty|)  \\  ty^* \mathrm{sinc}(|ty^*|)   &\cos(|ty|) 
\end{pmatrix}
\begin{pmatrix}
0 &  0  \\  0   &  \d1
\end{pmatrix}
\begin{pmatrix} \cos(|ty^*|) &     t \mathrm{sinc}(|ty^*|)y \\  t \mathrm{sinc}(|ty|)y^*   &\cos(|ty|) 
\end{pmatrix}.
$$ 
\end{rem}

\section{Embedded submanifolds}\label{submanifolds}

\subsection{Derivations induced by g1-pdms}

 We consider the following real Banach space: 
$$
i \ub= \left\{ \, \begin{pmatrix} x_1  &  x_2  \\   -\bar{x}_2  & - \bar{x}_1   \end{pmatrix} \in \cB(\fH\oplus \fH) :  x_1=x_1^*, \, x_2=-x_2^T \in  \cB_2(\fH)  \right\}\subseteq \cB_2(\fH\oplus\fH,D)
$$
equipped with the restricted norm $\| X \|_{\res}= 2\max\{ \|x_1\|, \, \|x_2\|_2 \}$. Noticing that each $\Gamma=\Gamma[\gamma, \alpha] \in \DF$ may be decomposed as 
$$
\Gamma=\begin{pmatrix}  \gamma  &  \alpha  \\  \alpha^*   &   \d1 -\bar{\gamma} \end{pmatrix}
= \begin{pmatrix}  \gamma  &  \alpha  \\  - \bar{\alpha}   &    -\bar{\gamma}\end{pmatrix} + \begin{pmatrix}  0   &   0  \\ 0  & \d1  \end{pmatrix}
:= \Gamma_0  + P_- , \, \, \, \, \, \Gamma_0 \in i \ub, 
$$
and since $\alpha \in \cB_2(\fH)$ by Remark \ref{d trace rem}, we have the inclusion $\cO(\Gamma) \subseteq i \ub+ P_-$.  It is then natural to ask for conditions on $\Gamma$ that characterize the homogeneous space $\cO(\Gamma)$ as an embedded submanifold of the affine space
$i \ub + P_-$.   We start with a necessary condition to be a submanifold according to Proposition \ref{crit immersion}, which consists in determining when tangent spaces  of the orbits are closed in the tangent space of the ambient manifold. Since  $\cO(\Gamma) \subseteq i \ub+ P_-$, the tangent space at $\Gamma$ satisfies the inclusion:
$$
T_\Gamma \cO(\Gamma) = \{   X\Gamma - \Gamma X    :   X \in \ub    \}  \subseteq i \ub.
$$
This lead us to introduce the following derivation.

\begin{fed}
Let $\Gamma \in \cD$. The \textit{derivation induced by $\Gamma$} is given by 
$$
\delta_\Gamma: \ub \to \ub , \, \, \, \, \delta_\Gamma(X)=[i\Gamma, X].
$$
\end{fed}
Observe $\Gamma=\Gamma_0 + P_-$, $\Gamma_0 \in i \ub$ and $X \in \ub$, then $[i\Gamma , X]=[i\Gamma_0,X] + [iP_- , X] \in \ub$. Hence the map $\delta_\Gamma$ actually takes values on $\ub$. It follows easily 
 that $\delta_\Gamma$ is continuous with $\| \delta_\Gamma\| \leq 2\| \Gamma \|_{\res}$, and it is a derivation of the Lie algebra $\ub$. Also notice   $i \ran(\delta_\Gamma)=T_\Gamma \cO(\Gamma)$. Thus, the tangent space $T_\Gamma \cO(\Gamma)$ is closed in $i\ub$ if and only if the derivation $\delta_\Gamma$ has closed range. 

We will use the following characterization of closed range linear maps in Banach spaces.

\begin{lem}\label{Banach closed range}
 Let $E$, $F$ be  Banach spaces and $T : E \to F$ a bounded linear map. Assume that there exists a closed subspace $\cM$ of $E$  such that $E=\mathrm{ker}(T) \oplus \cM$. Let $\cE$ be the continuous projection with $\ran(\cE)=\mathrm{ker}(T)$ and $\mathrm{ker}(\cE)=\cM$. 
Then, the range of $T$ is closed if and only if there exists $c>0$ such that $\| T e \| \geq c \|e - \cE e \|$, for all $e \in E$. 
\end{lem}
\begin{proof}
If the range of $T$ is closed, then $T_0:=T|_{\cM}:\cM \to \ran(T)$ is a bijective linear map, whose inverse is continuous by the open mapping theorem. Therefore, we have  $\| T e \|=\|T_0(e-\cE e)\| \geq \|T_0\|^{-1} \|e-\cE e\|$, for all $e \in E$. 
Conversely, take a sequence $\{ Te_n\}_{n \geq 1}$ in $\ran(T)$ such that $Te_n \to f$. 
Then,  $\|Te_n - Te_m\| \geq \| T((e_n-e_m) - \cE(e_n - e_m))\| \geq c \|e_n - \cE e_n - (e_m - \cE e_m)\|$, which implies that $\{e_n - \cE e_n \}_{n \geq 1}$ converges to some $e_0\in \cM$. Since $Te_n=T(e_n - \cE e_n)$, we obtain $f=Te_0$. Hence $\ran(T)$ is closed.
\end{proof}

\begin{prop}\label{closed range vs finite spectrum0}
Let $\Gamma \in \DF$. Then tangent spaces of $\cO(\Gamma)$ are closed in $i \ub$ if and only if  $\sigma(\Gamma)$ is finite. If any of these conditions hold, then $i\ub=i \ubg \oplus T_\Gamma \cO(\Gamma)$.
\end{prop}
\begin{proof}
All tangent spaces  are closed if and only the tangent space at some $\Gamma$ is closed. Hence we can assume that  $\Gamma=\Gamma[\Lambda, 0] \in \cD$ is diagonal, and study when the range of the corresponding derivation $\delta_\Gamma$ is closed. 

 Now suppose that $\sigma(\Gamma)$ is finite. According to Remark \ref{proj ubog diag}, this means that $\sigma(\Gamma)=\sigma( \Lambda) \cup \sigma(\d1 - \Lambda)= \{ \lambda_k : k=0, \ldots , r   \} \cup \{ 1 -\lambda_k : k=0, \ldots , r  \}$, for some $r< \infty$ ($\lambda_0=0$). Using the projections $\{ p_i\}_{i=0}^r$ in the mentioned remark: for $X=\begin{pmatrix}  x_1   &   x_2   \\   \bar{x}_2    &   \bar{x}_1 \end{pmatrix} \in \ub$, notice that
$$
[i\Lambda, x_1]  = \displaystyle{\sum_{\substack{i,j=0}}^r} i (\lambda_i -\lambda_j ) p_i x_1 p_j, \, \, \, \,
 i\Lambda x_2 -  x_2 i(\d1 - \Lambda) =  \sum_{i , j=0}^r i (\lambda_i +\lambda_j - 1 ) p_i x_2 p_j .
$$
Observe that $p_i [\Lambda, x_1] p_i=0$, for all $i=0, \ldots , r$, and when $\frac{1}{2}\in \sigma(\Gamma)$, $p_1(  \Lambda x_{2} -  x_{2}(\d1 - \Lambda))p_1=0$. 
If $\frac{1}{2} \notin \sigma(\Gamma)$, then  the  coefficients of the preceding sums satisfy    $\lambda_i - \lambda_j\neq 0 $, $i, j=0, \ldots, r$, $i\neq j$, and $\lambda_i + \lambda_j -1\neq 0$, $i, j=0, \ldots, r$. Thus, the range of $\delta_\Gamma$ can be expressed as
\begin{equation}\label{tangent space as range case 1}
\ran(\delta_\Gamma)={\small \left\{ 
\begin{pmatrix}    \displaystyle{\sum_{\substack{i,j=0\\ i \neq j}}^r } p_j x_1 p_i  &    x_2  \\
 \bar{x}_2   &   \displaystyle{\sum_{\substack{i,j=0\\ i \neq j}}^r  p_j \bar{x}_1 p_i}   \end{pmatrix} 
:  x_1 = -x_1^* \in \cB(\fH), \, x_2=-x_2^T \in \cB_2(\fH) \right\}}.
\end{equation}
In the case where $\frac{1}{2} \in \sigma(\Gamma)$,  the above expression should be modified: 
\begin{equation}\label{tangent space as range case 2}
\ran(\delta_\Gamma)={\small \left\{ 
\begin{pmatrix}    \displaystyle{\sum_{\substack{i,j=0\\ i \neq j}}^r } p_j x_1 p_i  &   \displaystyle{\sum_{\substack{i,j=0\\ (i ,j)\neq (1,1)}}^r } p_j x_2 p_i \\
\displaystyle{\sum_{\substack{i,j=0\\ (i,j)\neq (1,1)}}^r } p_j \bar{x}_2 p_i &   \displaystyle{\sum_{\substack{i,j=0\\ i \neq j}}^r p_j \bar{x}_1 p_i}   \end{pmatrix} 
:  x_1 = - x_1^* \in \cB(\fH), \, x_2=-x_2^T \in \cB_2(\fH) \right\}}.
\end{equation}
In both cases, $\ran(\delta_\Gamma)$ is clearly a closed subspace of $ \ub$.

For the converse we use the  characterization in Lemma \ref{Banach closed range}. If we assume that $\ran(\delta_\Gamma)$ is closed in $\ub$, then there is a constant $c>0$ satisfying 
\begin{equation}\label{closed range00}
\| \delta_\Gamma (X)\|_{\res} \geq c  \, \| X - \cE_\Gamma(X)  \|_{\res},  \, \, \, \, X \in \ub.
\end{equation}
Here $\cE_\Gamma$  is the conditional expectation with  $\ran(\cE_\Gamma)=\ker(\delta_\Gamma)=\ubg$ and   $\ker(\cE_\Gamma)=\mathfrak{m}_\Gamma$ is the supplement  defined in the proof of Lemma \ref{lemma Lie subg g1-pdms}. Let  $\cB(\fH)_{\mathrm{sk}}$ denote the real space of  bounded linear  skew-adjoint operators on $\fH$. Consider the derivation $\delta_\Lambda^{\mathrm{sk}}:=\delta_\Lambda|_{\cB(\fH)_{\mathrm{sk}}}: \cB(\fH)_{\mathrm{sk}} \to \cB(\fH)_{\mathrm{sk}}$, $\delta_\Lambda^{\mathrm{sk}}(x)=[i\Lambda,x]$. 
Then take a operators $X=\diag(x, \bar{x})$, $x^*=-x$, in \eqref{closed range00} to find that
$\| \delta_\Lambda^{\mathrm{sk}}(x) \|   \geq c \| x- \cE_\Lambda(x)\|$,  for all $x=-x^*$,
where  $\cE_\Lambda$ is the continuous projection such that $\ran(\cE_\Lambda)=\ker(\delta_\Lambda^{\mathrm{sk}})$ and  $\ker(\cE_\Lambda)=(\mathfrak{m}_\Lambda)_{\mathrm{sk}}:=\{ x \in \cB(\fH)_\mathrm{sk} :  x=x  - \sum_{i=0}^r p_i x_1 p_i  \}$, 
which is a closed supplement of $\ker(\delta_\Lambda^{\mathrm{sk}})$. 
  This shows that $\delta_\Lambda^{\mathrm{sk}}$  has closed range, which implies that $\delta_\Lambda: \cB(\fH) \to \cB(\fH)$, $\delta_\Lambda(x)=[i \Lambda, x]$ also has closed range since $\Lambda^*=\Lambda$. This is equivalent to have that $\sigma(\Lambda)$ is finite by \cite[Thm. 3.3]{AF75}, which means that $\sigma(\Gamma)$ is finite.

Finally, the expressions in \eqref{tangent space as range case 1} and \eqref{tangent space as range case 2} imply that, under the assumption that $\sigma(\Gamma)$ is finite, we have $\ran(\delta_\Gamma)=\ker(\cE_\Gamma)$. Hence $\ub=\ran(\cE_\Gamma)   \oplus \ker(\cE_\Gamma)=\ubg \oplus \ran(\delta_\Gamma)$, o equivalently,
  $i\ub=i \ubg \oplus T_\Gamma \cO(\Gamma)$.
\end{proof}

\begin{rem}
Furthermore, we will show later that $\Tr(X\delta_\Gamma(Y))=0$, for all $Y\in \ub$, if and only if $X\in \ubg$  (see Lemma \ref{lema 2 symplectic}). 
\end{rem}

\subsection{Embedded submanifod structure of orbits of g1-pdms}

Now we consider another necessary condition to be an embedded submanifold. Since we know that $\cO(\Gamma) \subseteq i \ub + P_-$, the orbits can be endowed with the relative topology inherited from the affine space $i \ub + P_-$. This topology is defined by the metric $d(\Gamma_0, \Gamma_1)=\| \Gamma_0 - \Gamma_1\|_{\res}$, for $\Gamma_0 , \, \Gamma_1 \in \cO(\Gamma)$. On the other hand, $\cO(\Gamma) \cong \UB / \UB^\Gamma$ 
is a homogeneous space of $\UB$, so that the orbits can be endowed with the quotient topology. Let us write $\tau_r$ and $\tau_q$ for the relative and  quotient topology, respectively. Both topologies must coincide when the orbits are embedded submanifolds of $i \ub + P_-$.  We will see that the coincidence of these topologies is again related to the finite spectrum condition.

Recall that the map $\pi_\Gamma:\UB \to \cO(\Gamma)$, $\pi_\Gamma(U)=U\Gamma U^*$ is a submersion when $\cO(\Gamma)$ is considered as a homogeneous space, which in particular implies that $\pi_\Gamma$ has  continuous local cross sections. From this fact, it follows that $\tau_q$ is stronger than $\tau_r$. Then, both topologies coincide if and only if  the map $\pi_\Gamma: \UB \to \cO(\Gamma)$, $\pi_\Gamma(U)=U\Gamma U^*$ admits continuous local cross sections when $\cO(\Gamma)$ is considered with the topology $\tau_r$.  To construct such a local cross section we essentially follow the argument in \cite[Thm. 4.4]{AL10}, but it needs some additional work because the restricted norm is not isometric by multiplication with Bogoliubov transformations.

\begin{rem}
We first need to extend the derivation $\delta_\Gamma$ and the conditional expectation $\cE_\Gamma$.

\medskip

\noi $i)$ Given $\Gamma \in \cD$, set 
$$
\tde: \cB_2(\fH\oplus \fH, D)  \to  \cB_2(\fH\oplus \fH, D), \, \, \, \, \tde(X)=[i\Gamma, X].
$$
It is straightforward to check that $\tde$ is  a continuous derivation of the Lie algebra $\cB_2(\fH \oplus \fH, D)$.  For $\Gamma=\Gamma[\Lambda, 0]\in \cD$ a diagonal g1-pdm and using the same notation of  Definition  \ref{cond exp} for the spectral decomposition of $\Lambda$ ($1 \leq r \leq \infty$), we define the linear projection 
$$\tilde{\cE}_\Gamma: \cB_2(\fH \oplus \fH, D) \to   \cB_2(\fH \oplus \fH, D),$$
\begin{equation*}
\tilde{\cE}_\Gamma\left(   \begin{pmatrix}  x_{11}  &   x_{12}   \\   x_{21}  &   x_{22} \end{pmatrix}   \right)  =  
\begin{cases}
      \begin{pmatrix} \sum_{i=0}^r p_i x_{11} p_i   &   0  \\
0  &  \sum_{i=0}^r p_i x_{22} p_i 
\end{pmatrix} &   \text{ if } \frac{1}{2}  \notin \sigma(\Gamma),\\
      \begin{pmatrix} \sum_{i=0}^r p_i x_{11} p_i   &   p_1 x_{12} p_1  \\
p_1 x_{21} p_1  &  \sum_{i=0}^r p_i x_{22} p_i 
\end{pmatrix} &  \text{ if } \frac{1}{2}  \in \sigma(\Gamma).
					\end{cases}       
\end{equation*}
We fix the following notation
$
\tilde{\mathfrak{m}}_\Gamma =\ker(\tilde{\cE}_\Gamma).
$

\medskip

\noi $iii)$ We denote by  $\delta_\Gamma^0$ and $\cE_\Gamma^0$ for the extensions in the obvious way of $\delta_\Gamma$ and $\cE_\Gamma$ to the whole $\cB(\fH\oplus \fH)$. We also put $\mathfrak{m}_\Gamma^0=\ker(\delta_\Gamma^0)$.
\end{rem}

One can show by similar arguments to that of Proposition \ref{closed range vs finite spectrum0} that the above derivations $\tde$ and $\delta_\Gamma^0$ also have closed range if and only if 
$\sigma(\Gamma)$ is finite. Rather than proving this fact, we focus  on estimate the constant given by the characterization in Lemma \ref{Banach closed range} in terms of the eigenvalues of $\Gamma$. 
This will be helpful  later for the construction of  continuous local cross sections.  

\begin{lem}\label{constants two types}
Let $\Gamma=\Gamma[\Lambda,0] \in \DF$ be a diagonal g1-pdm with finite spectrum. Then the following assertions hold:
\begin{itemize}
\item[i)] For every $X \in \cB_2(\fH \oplus \fH, D)$, we have $
\|   \tde(X) \|_{\res} \geq \tilde{c}_\Gamma \| X   - \tilde{\cE}_\Gamma (X)\|_{\res}$, 
where the constant $\tilde{c}_\Gamma$ can be taken as
\begin{equation*}
\tilde{c}_\Gamma =    
\begin{cases}
      \min \{   (\sum_{i\neq j}|\lambda_i - \lambda_j|^{-1})^{-1} , (\sum_{i,j=0}^r| \lambda_i + \lambda_j -1|^{-1})^{-1}  \} &   \text{ if } \frac{1}{2}  \notin \sigma(\Gamma),\\
      \min \{   (\sum_{i\neq j}|\lambda_i - \lambda_j|^{-1})^{-1} , (\sum_{i,j=0, (i,j)\neq (1,1)}^r| \lambda_i + \lambda_j -1|^{-1})^{-1}  \} &  \text{ if } \frac{1}{2}  \in \sigma(\Gamma).
			\end{cases}       
\end{equation*}
\item[ii)]  For every $X \in \cB(\fH \oplus \fH)$, we have $
\|   \delta_\Gamma^0(X) \| \geq \tilde{c}_\Gamma^0 \| X   - \tilde{\cE}_\Gamma^0 (X)\|$, where
\begin{equation*}
c_\Gamma^0 =    
\begin{cases}
       \frac{1}{2}\left(\sum_{i\neq j}|\lambda_i - \lambda_j|^{-1} +   \sum_{i,j=0}^r| \lambda_i + \lambda_j -1|^{-1}\right)^{-1}&   \text{ if } \frac{1}{2}  \notin \sigma(\Gamma),\\
        \frac{1}{2}\left(\sum_{i\neq j}|\lambda_i - \lambda_j|^{-1} +   \sum_{i,j=0, (i,j)\neq (1,1)}^r| \lambda_i + \lambda_j -1|^{-1}\right)^{-1} &  \text{ if } \frac{1}{2}  \in \sigma(\Gamma).
			\end{cases}       
\end{equation*}
\end{itemize}
\end{lem}
\begin{proof}
$i)$ For $X=\begin{pmatrix}  x_{11}   &   x_{12}   \\   x_{21}    &   x_{22} \end{pmatrix} \in \cB_2(\fH \oplus \fH, D)$, the block operator entries of $\tilde{\delta}_\Gamma(X)=[i\Gamma,X]$ are given by
\begin{align*}
& [i\Lambda, x_{11}]  = \displaystyle{\sum_{\substack{i,j=0}}^r} i (\lambda_i -\lambda_j ) p_i x_{11} p_j \, , \, \, \, \, \, \, 
  i\Lambda x_{12} -  x_{12}i(\d1 - \Lambda) =  \sum_{i , j=0}^r i (\lambda_i +\lambda_j - 1 ) p_i x_{12} p_j \\
& i(\d1 - \Lambda)x_{21}- x_{21}i \Lambda   =  \sum_{i , j=0}^r     i(1-\lambda_i -\lambda_j  )  p_j x_{21} p_i    \, , \, \, \, \, \, \, 
[i(\d1 -\Lambda),x_{22}]  = \displaystyle{\sum_{\substack{i,j=0}}^r} i (\lambda_j -\lambda_i ) p_i x_{22} p_j \, .
\end{align*}
Motivated by these formulas, for each $x \in \cB(\fH)$, when $\frac{1}{2}\notin \sigma(\Gamma)$, set
$$
A_{x}=-\displaystyle{\sum_{\substack{i,j=0\\ i \neq j}}^r} i (\lambda_i -\lambda_j )^{-1} p_i x p_j , \, \, \, \, 
B_{x}=-   \sum_{i , j=0}^r i(\lambda_i  +\lambda_j -1 )^{-1} p_i x p_j  .
$$
In the case where $\frac{1}{2}\in \sigma(\Gamma)$, the definition of $A_x$ does not change, and for $B_x$ we now  set
$$
B_{x}=-\displaystyle{   \sum_{\substack{i , j=0 \\ (i ,j)\neq (1,1)}}^r}  i(\lambda_i  +\lambda_j -1 )^{-1} p_i x p_j .
$$
In both cases define $\beta_\Gamma : \cB_2(\fH\oplus \fH, D) \to  \cB_2(\fH \oplus \fH,D)$, 
$$
\beta_\Gamma\left(\begin{pmatrix}  x_{11}   &   x_{12}   \\   x_{21}    &   x_{22} \end{pmatrix}\right)=\begin{pmatrix}
A_{x_{11}}  &      B_{x_{12}} \\   -B_{x_{21}}   &    - A_{x_{22}}
\end{pmatrix}.
$$ 
This is a bounded linear operator acting on $\cB_2(\fH\oplus \fH,D)$ satisfying $\beta_\Gamma \circ \tde=\tde \circ \beta_\Gamma=\d1 - \tilde{\cE}_\Gamma$ (here $\d1=\d1_{\cB_2(\fH \oplus \fH, D)}$).  
Therefore, $\|  X   - \tilde{\cE}_\Gamma (X) \|_{\res} =\| (\beta_\Gamma \circ \tde)(X)\|_{\res}\leq \| \beta_\Gamma \| \| \tde(X)\|_{\res}$, which yields $\| \tde(X) \|_{\res} \geq \| \beta_\Gamma \|^{-1} \|  X   - \tilde{\cE}_\Gamma (X) \|_{\res}$ for all $X$. If $\frac{1}{2} \notin \sigma(\Gamma)$, then one can see $\| \beta_\Gamma\|\leq \max \{   \sum_{i\neq j}|\lambda_i - \lambda_j|^{-1} , \sum_{i,j=0}^r| \lambda_i + \lambda_j -1|^{-1} \} $. This implies that we can take
 $\tilde{c}_\Gamma=\min \{   (\sum_{i\neq j}|\lambda_i - \lambda_j|^{-1})^{-1} , (\sum_{i,j=0}^r| \lambda_i + \lambda_j -1|^{-1})^{-1}  \}.$ 
For the case where $\frac{1}{2} \in \sigma(\Gamma)$ we only need to write $\sum_{i,j=0, (i,j)\neq (1,1)}^r| \lambda_i + \lambda_j -1|^{-1}$ instead of $\sum_{i,j=0}^r| \lambda_i + \lambda_j -1|^{-1}$.

\medskip

\noi $ii)$ Similar arguments apply for the extensions $\delta_\Gamma^0$ and $\cE_\Gamma^0$ when one takes the operator norm $\|  \, \cdot \,\|$ instead of the restricted norm $\|  \, \cdot \, \|_{\res}$. 
\end{proof}

In general, the restricted norm of a Bogoliubov transformation can have arbitrary large restricted norm. In the following lemma, we control the restricted norm  of Bogoliubov transformations acting on g1-pdms that are close enough in the relative topology.

\begin{lem}\label{technical lemma 00}
Suppose that $\Gamma=\Gamma[\Lambda,0] \in \DF$ is a diagonal g1-pdm with finite spectrum. Then there exists a constant $K > 0$, depending only on $\Gamma$, such that $\|U\|_{\res} \leq K$, for all $U \in \UB$ satisfying $\| U\Gamma U^* - \Gamma \|_{\res} \leq \frac{c_\Gamma^0}{3}$.
\end{lem}
\begin{proof}
Using that $\Gamma$ is  diagonal and expression in \eqref{action ubog}, we note 
$$
U=\begin{pmatrix}  u  &  v  \\   \bar{v} &   \bar{u}   \end{pmatrix}, \, \, \, \, \, U\Gamma U^*=\begin{pmatrix}   *   &  
u\Lambda \bar{v}^*  +  v(\d1-\Lambda)\bar{u}^*       \\  *   &  *  \end{pmatrix}.
$$
Recall that $\|U\|_{\res}=2 \max\{ \|u\|, \, \|v\|_2 \}$. The operator $U$ is unitary,  which implies that  $\|u\| \leq 1$ and $\|v\| \leq 1$. Thus, our task is to estimate $\|v\|_2$.  Moreover, notice that $\Gamma \in \DF$ with finite spectrum says that  $\Lambda$ is a finite-rank operator. Using the notation in Remark \ref{proj ubog diag}, we have $\Lambda=\sum_{i=0}^r \lambda_i p_i$, where $\sum_{i=0}^rp_i=\d1$ and $p_0$ is the unique projection with infinite rank. Since $v=p_0vp_0  +  (\d1 -p_0)vp_0 + v(\d1 - p_0)$ and $\d1 - p_0$ has finite rank, we only have to estimate $\|p_0vp_0\|_2$. By assumption $\| U\Gamma U^* - \Gamma \|_{\res} \leq  \frac{c_\Gamma^0}{3}$, and therefore, 
$\| u\Lambda \bar{v}^*  +  v(\d1-\Lambda)\bar{u}^* \|_2    \leq  \frac{c_\Gamma^0}{6}$. Since $\Lambda \in \cB_2(\fH)$, we obtain
\begin{align*}
\|p_0 v p_0 \bar{u}^*\|_2   &    \leq    \frac{c_\Gamma^0}{6}  + \|u\Lambda \bar{v}^*\|_2      +   \|  v\Lambda\bar{u}^*  \|_2  +  \|  (\d1 -p_0)vp_0 + v(\d1 - p_0) \bar{u}^* \|_2 \\
&  \leq   \frac{c_\Gamma^0}{6} +  \|u\| \|\bar{v}^*\| \|\Lambda \|_2  +   \| v\| \| \bar{u}^* \|   \|\Lambda \|_2    +      \|  (\d1 -p_0)\|_2 \| vp_0 \|+  \|v\| \|(\d1 - p_0)\|_2 \| \bar{u}^* \|     \\
& \leq \frac{c_\Gamma^0}{6} + 2 \| \Lambda \|_2  + 2\|  \d1  - p_0 \|_2 :=C_1 . 
\end{align*}
Next consider
\begin{align*}
\|p_0 v p_0 \bar{u}^*\|_2^2  & =\| \bar{u}p_0 v^* p_0\|_2^2 = \Tr(p_0v p_0\bar{u}^*\bar{u} p_0v^*p_0)\\ 
& =\Tr(p_0v p_0(\d1 - \bar{v}^*\bar{v}) p_0v^*p_0) \geq (1- \|  p_0 \bar{v}^*\bar{v}p_0\|) \|  p_0vp_0 \|_2^2
\end{align*}
 which gives
\begin{equation*}
 (1- \|  p_0 \bar{v}^*\bar{v}p_0\|) \|  p_0vp_0 \|_2^2 \leq   \|p_0vp_0 \bar{u}^*\|_2^2 \leq C_1^2.
\end{equation*}
We will see below that $\|p_0 \bar{v}^*\bar{v} p_0\|=\|p_0 v^*v p_0\| \leq C<1$ for some $C>0$. This would imply the desired estimate:
\begin{equation}\label{Tr HS estimate}
\|p_0vp_0\|_2  \leq (1-C^2)^{-1/2}C_1.
\end{equation}
To get a finer estimate the operator norm of $v$,  we first give a  bound for the operator norm in terms of the restricted norm:  for $X \in \cB_2(\fH \oplus \fH,D)$,
\begin{align*}
\| X \| & \leq \|  P_+XP_+\|  +\|P_+ X P_-  \|  +  \|P_- X  P_+  \|  +\|P_- X P_-  \| \\
& \leq  \|  P_+XP_+\|  +\|P_+ X P_-  \|_2  +  \|P_- X  P_+  \|_2  +\|P_- X P_-  \| \leq 2 \|  X  \|_{\res}.
\end{align*} 
In the case where $\frac{1}{2}\notin \sigma(\Gamma)$, then by the previous estimate and Lemma \ref{constants two types} $ii)$  we obtain 
\begin{align*}
\| v p_0 \| & \leq \| v\| =\| P_+ (U- \cE_\Gamma^0(U))P_- \|\leq \|  U- \cE^0_\Gamma(U) \|  \leq (c_\Gamma^0)^{-1} \|  \delta_\Gamma^0(U)\| \\
& =(c_\Gamma^0)^{-1} \|  U\Gamma U^*   - \Gamma \|  \leq 2  (c_\Gamma^0)^{-1} \|  U\Gamma U^*   - \Gamma \|_{\res} \leq \frac{2}{3}.
\end{align*}
In the case where $\frac{1}{2} \in \sigma(\Gamma)$ we have $ P_+ (U- \cE_\Gamma^0(U))P_-=v - p_1vp_1$, and in the similar fashion as above we get
$\|vp_0\| = \|(v - p_1 v p_1)p_0    \| \leq \| v - p_1 v p_1 \|  =\|  P_+ (U- \cE_\Gamma^0(U))P_-   \| \leq\frac{2}{3} $.
In both cases, $\|  p_0v^*vp_0 \|=\| vp_0 \|^2< \frac{4}{9}$, so by Eq. \eqref{Tr HS estimate} we conclude $\|p_0vp_0 \|_2 \leq \frac{9}{\sqrt{65}} C_1$. 
Hence we conclude $\| v\|_2 =\| p_0vp_0  +  (\d1 -p_0)vp_0 + v(\d1 - p_0)\|_2 \leq  \frac{9}{\sqrt{65}} C_1 + 2 \| \d1   -  p_0 \|_2 :=K$, where this constant only depends on $\Gamma$.
\end{proof}

 
We are now able to prove our main result on the two topologies of orbits of g1-pdms.

\begin{prop}\label{closed top vs finite spectrum0}
Let $\Gamma \in \DF$. Then $\tau_r=\tau_q$ if and only if $\sigma(\Gamma)$ is finite.
\end{prop}
\begin{proof}
Again we may assume that  $\Gamma=\Gamma[\Lambda, 0]$ is a diagonal g1-pdm. Suppose that $\tau_r=\tau_q$. If $\sigma(\Gamma)$ is not finite, then the derivation $\delta_\Gamma: \ub \to  \ub$, $\delta_\Gamma(X)=[i \Gamma, X]$ has non-closed range by Proposition \ref{closed range vs finite spectrum0}. This is equivalent to the existence of a sequence $\{ X_n\}_{n \geq 1}$ in $\mathfrak{m}_\Gamma$ such that $\|X_n\|_{\res}=1$ and $\| \delta_\Gamma(X_n)\|_{\res} \leq 1/n$. Recalling the well-known formula 
$e^{Z}\Gamma e^{-Z}=\Ad_{e^{Z}}(\Gamma)=e^{\delta_Z}(\Gamma)=(\d1_{\cB(\fH\oplus \fH)} + \delta_Z + \delta_Z^2/2 + \delta_Z ^3/3! + \ldots)(\Gamma)$, where $\delta_Z(X)=[iZ,X]$, and since $\| \delta_{X_n}\| \leq 2 $, we find that
\begin{align*}
\|  e^{X_n} \Gamma e^{-X_n}      -  \Gamma \|_{\res} & = \|  e^{\delta_{X_n}}(\Gamma) - \Gamma \|_{\res} 
\leq  \sum_{k=1}^{\infty}  \frac{\|\delta_{X_n}^k(\Gamma)\|_{\res}}{k!}  \\
&  \leq  \sum_{k=1}^{\infty}  \frac{\|\delta_{X_n}\|^{k-1}\|\delta_\Gamma(X_n)\|_{\res}}{k!} \leq \frac{1}{n} \sum_{k=1}^\infty \frac{2^{k-1}}{k!} \leq \frac{e^2}{n}.
\end{align*}

That is, the sequence $\{  e^{X_n} \Gamma e^{-X_n}\}_{n \geq 1}$ converges to $\Gamma$ in the topology $\tau_r$, and consequently, it also converges in the topology $\tau_q$. The maps $\phi: \cV_\Gamma \to \cW_\Gamma$, $\phi(Z)=e^Z\Gamma e^{-Z}$, where $0 \in \cV_\Gamma \subset \mathfrak{m}_\Gamma$ and $\Gamma \in \cW_\Gamma\subseteq \cO(\Gamma)$ are open sets, become homeomorphisms when these sets are sufficiently small and $\cO(\Gamma)$ is endowed with the topology $\tau_q$. Indeed, these maps are the inverse of the charts used to provide the manifold structure of 
$\cO(\Gamma)$ as homogeneous space of $\UB$ (see \cite[Lemma 4.21]{B06}). In particular, this yields that $\| X_n\|_{\res} \to 0$, a contradiction. Hence  $\sigma(\Gamma)$ must be finite.

Let us assume, conversely,  that  $\sigma(\Gamma)$ is finite. Suppose that $\Gamma=\Gamma[\Lambda, 0]$ is a diagonal g1-pdm.  In order to show that both topologies coincide, we will  prove that  $\pi_\Gamma: \UB \to \cO(\Gamma)$, $\pi_\Gamma(U)=U\Gamma U^*$, admits continuous local cross sections when $\cO(\Gamma)$ is considered with the topology $\tau_r$. 
For $G$ an invertible operator, let $\Omega(G)=G|G|^{-1}$ be the unitary part in the polar decomposition. Consider $c_\Gamma^0$, $\tilde{c}_\Gamma$ and $K$ the positive constants  defined in   Lemmas \ref{constants two types} and \ref{technical lemma 00}.  Then we set $r_\Gamma:=\frac{1}{2}\min\left\{  \frac{c_\Gamma^0}{3}, \,  K^{-2} \tilde{c}_\Gamma \right\}$. The section $s=s_\Gamma$ at $\Gamma$ is then defined by
$$
s: \left\{  U\Gamma U^*  \in \cO(\Gamma):    \|  U\Gamma U^*  - \Gamma\|_{\res}< r_\Gamma  \right\}   \to  \UB, \, \, \, \, \, s(U\Gamma U^*)= U\Omega(\tilde{\cE}_\Gamma(U^*)).
$$  
 Clearly, we have $ U\Omega(\tilde{\cE}_\Gamma(U^*)) \in \cB_2(\fH \oplus \fH, D)$, for  $U \in \UB$.  We claim that if $\|  U\Gamma U^*  - \Gamma\|_{\res}< r_\Gamma$, then $\tilde{\cE}_\Gamma(U^*)$ is  invertible in the algebra $\cB_2(\fH \oplus \fH,D)$. To see this we recall that by our choice of $r_\Gamma$,  Lemma \ref{technical lemma 00}  implies $\|U\|_{\res}\leq K$,  or equivalently, $\|U^*\|_{\res}\leq K$. Therefore, for $\d1=\d1_{\fH \oplus \fH}$,
\begin{align}
\| \d1 - U\tilde{\cE}_\Gamma(U^*)\|_{\res} &  \leq \| U\|_{\res} \| U^* - \tilde{\cE}_\Gamma(U^*)\|_{\res} \leq K \tilde{c}_\Gamma^{-1}  \| \tilde{\delta}_\Gamma(U^*)   \|_{\res} \nonumber \\ 
& \leq  K \tilde{c}_\Gamma^{-1} \|U^*  \|_{\res}\|  U\Gamma U^*   - \Gamma  \|_{\res}   \leq  K^2 \tilde{c}_\Gamma^{-1}\| U\Gamma U^*   - \Gamma  \|_{\res} < 1.  \nonumber 
\end{align} 
Hence $U\tilde{\cE}_\Gamma(U^*)$  invertible in $\cB_2(\fH\oplus \fH,D)$, and then $\tilde{\cE}_\Gamma(U^*)$ is also invertible. This proves our claim.

Now we check that $s$ is well defined. Suppose that $\Gamma_1=U\Gamma U^*=V\Gamma V^*$ for $U, V \in \UB$. 
Thus, $U^*V\Gamma=\Gamma U^*V$, which gives
\begin{align*}
s(U\Gamma U^*) & =U\Omega(\tilde{\cE}_\Gamma(U^*))=U\Omega(\tilde{\cE}_\Gamma(U^*VV^*))=U\Omega(U^*V \tilde{\cE}_\Gamma(V^*)) \\
& =UU^*V \Omega(\tilde{\cE}_\Gamma(V^*))=V\Omega(\tilde{\cE}_\Gamma(V^*))=s(V\Gamma V^*).  
\end{align*}
Here we have used that $\tilde{\cE}_\Gamma(ZX)=Z\tilde{\cE}_\Gamma(X)$, whenever $X, Z \in \cB_2(\fH\oplus \fH,D)$ and $Z\Gamma=\Gamma Z$.

Now note that $s$ actually takes values on $\UB$. 
From the definition of $\tilde{\cE}_\Gamma$,  $\tilde{\cE}_\Gamma(U^*)I=I\tilde{\cE}_\Gamma(U^*)$ because $U \in \UB$, which yields $\Omega(\tilde{\cE}_\Gamma(U^*))I=I\Omega(\tilde{\cE}_\Gamma(U^*))$. It is not difficult to check  that $\ran(\tilde{\cE}_\Gamma)$ is a closed $*$-subalgebra of $\cB_2(\fH \oplus \fH,D)$, so that $\Omega(\tilde{\cE}_\Gamma(U^*)) \in \cB_2(\fH \oplus \fH,D)$ by using the analytic functional calculus in the Banach algebra $\cB_2(\fH \oplus \fH)$. From these facts, we conclude that $\Omega(\tilde{\cE}_\Gamma(U^*)) \in \UB$, and consequently,  $s(U\Gamma U^*) \in \UB$.  On the other hand, $s$ is indeed a section. To see this, take $\Gamma_1=U\Gamma U^*$. Since $\tilde{\cE}_\Gamma(U^*)\Gamma=\Gamma \tilde{\cE}_\Gamma(U^*)$, we find that $\Omega(\tilde{\cE}_\Gamma(U^*))\Gamma=\Gamma \Omega(\tilde{\cE}_\Gamma(U^*))$. Hence $s(\Gamma_1)\Gamma s(\Gamma_1)^*=U \Omega(\tilde{\cE}_\Gamma(U^*)) \Gamma \Omega(\tilde{\cE}_\Gamma(U^*))^* U^*=U\Gamma U^*= \Gamma_1$. 

We now prove that $s$ is continuous.  It suffices to prove the continuity at $\Gamma$. Pick a sequence $\{ U_n \Gamma U_n^*\}_{n \geq 1}$ such that $\| U_n\Gamma U_n^* - \Gamma \|_{\res} \to 0$. By the estimate we have seen above, it follows that $\| \d1 - U_n \tilde{\cE}_\Gamma (U_n^*) \|_{\res} \leq  K^2 c_\Gamma^{-1}\| U_n\Gamma U_n^*   - \Gamma  \|_{\res} \to 0$. Using that $\Omega(G)=G|G|^{-1}$ is  continuous as a map in the invertible elements of the Banach $*$-algebra $\cB_2(\fH \oplus \fH)$ by the continuity  of the analytic functional calculus, we get $s(U_n \Gamma U_n^*)=U_n \Omega(\tilde{\cE}_\Gamma(U_n^*))=\Omega(U_n \cE_\Gamma(U_n^*)) \to \Omega(\d1)=\d1=s(\Gamma)$.  
Finally, we note that section  can be  translated to other points by putting $s_{\Gamma_1}=\Ad_U  \circ s_\Gamma \circ \Ad_{U^*}$ if $\Gamma_1=U\Gamma U^*$.  
\end{proof}

As a consequence of the previous results we now obtain the following characterization.

\begin{teo}\label{submanifold}
Let $\Gamma \in \DF$. The following conditions are equivalent:
\begin{enumerate}
\item[i)] The spectrum of $\Gamma$ is finite.
\item[ii)] Tangent spaces of $\cO(\Gamma)$ are closed in $i \ub$.
\item[iii)] The topology inherited from $i \ub + P_-$ and the quotient topology coincide in $\cO(\Gamma)$.
\item[iv)] $\cO(\Gamma)$ is an embedded submanifold of the affine space $i \ub + P_-$. 
\end{enumerate}
\end{teo}
\begin{proof}
The equivalences between $i)$, $ii)$ and $iii)$ are proved in Propositions \ref{closed range vs finite spectrum0} and \ref{closed top vs finite spectrum0}. According to Proposition   \ref{crit immersion}, these items are  equivalent to item  $iv)$, since Proposition   \ref{closed range vs finite spectrum0} also gives  that tangent spaces are complemented subspaces in $i \ub$. 
\end{proof}




\section{ K\"ahler homogeneous spaces}\label{symplectic complex kahler}

\subsection{An invariant  symplectic form on orbits of g1-pdms}\label{symplectic structures}



In this section, we will  construct an invariant symplectic form on orbits of g1-pdms.

\begin{rem}
We will frequently use without mention the following well-known fact for computations involving the trace. If $x$ is a compact operator and
$y$ a bounded operator such that $xy$ and $yx$ are trace-class operators, then $\Tr(xy)=\Tr(yx)$ (see \cite[Chap.III Thm. 8.2]{GK60}).
\end{rem}

\begin{rem}
\noi $i)$ Consider the complex Lie algebra $\fg=\mathfrak{gl}_2(\fH \oplus \fH,I,D)$ defined in Remark \ref{rem Banach alg-restricted orthogonal-isomorphism}.  This corresponds to the case $k_\infty=1$ in \cite[Prop. I.11]{N} (see also \cite[Ex. I.9.b]{N}), which says that  $H^2_c(\fg, \C)\cong \C$. 
As we have already observed in the previous section, $\fg_\R=\{ X \in \fg : X^*=-X \}=\ub$. Then same arguments to that of \cite[Prop. 2.4]{BRT} give 
that $H^2_c(\ub,\R)\cong \R$. Further, a generator of $H^2_c(\ub,\R)$ is given by
\begin{equation*}
s_+(X,Y):=\Tr(X [iP_+,Y]), \, \, \, \, \, X,Y \in \ub. 
\end{equation*}
This can also be expressed as
$$
s_+\left(\begin{pmatrix} x_1    &    x_2   \\  \bar{x}_2   &     \bar{x}_1   \end{pmatrix},  \begin{pmatrix} y_1    &    y_2   \\  \bar{y}_2   &     \bar{y}_1   \end{pmatrix} \right)=2 \Im \Tr(x_2\bar{y}_2).
$$

\medskip

\noi $ii)$ It is interesting to point out that  a multiple of the above generator already appeared in the literature as a $2$-cocycle of $\mathfrak{o}_{\mathrm{res}}(\fH^\R)$. Following the notation in Remark \ref{rem Banach alg-restricted orthogonal-isomorphism} $ii)$, recall that  for $A \in \cB(\fH^\R)$, we write $A_a$ for its antilinear part.   It was proved by Vershik \cite{V} that the function
$$
\alpha:\mathfrak{o}_{\mathrm{res}}(\fH^\R) \times \mathfrak{o}_{\mathrm{res}}(\fH^\R) \to \R, \, \, \, \, \, \, \alpha(A,B)=\Tr([A_a ,B_a]J_0)=2 \Tr(A_a B_a J_0),
$$
is a continuous real 2-cocycle of $\mathfrak{o}_{\mathrm{res}}(\fH^\R)$. The trace is taken using  the inner product $\Re  \PI{\, }{\, }$ of $\fH^\R$; $A_a$, $B_a$ belong to $\cB_2(\fH^\R)$ if $A,B \in \mathfrak{o}_{\mathrm{res}}(\fH^\R)$, which leads to the different expressions above for $\alpha$. 
The differential map at the identity $\d1=\d1_{\fH\oplus \fH}$ of the Lie group isomorphism in \eqref{isom Banach alg} is given by $(d \, \Xi)_{\d1}:  \mathfrak{o}_{\mathrm{res}}(\fH^\R) \to \ub$, $(d \, \Xi)_{\d1}(A)=T^{-1}A_\C T$, which is  a Lie algebra isomorphism.  Therefore  there is a  continuous real 2-cocycle  $ \Xi_*\alpha: \ub \times \ub \to \R$ defined by 
$$
(\Xi_* \alpha)(X,Y):=\alpha((d \, \Xi)_{\d1}^{-1}(X), (d \, \Xi)_{\d1}^{-1}(Y)).
$$
Using analogous relations for the algebras level to the ones stated  in \eqref{iso matrix op}  for the group level, it is easy to see that 
$$
(\Xi_* \alpha)\left(\begin{pmatrix} x_1    &    x_2   \\  \bar{x}_2   &     \bar{x}_1   \end{pmatrix},  \begin{pmatrix} y_1    &    y_2   \\  \bar{y}_2   &     \bar{y}_1   \end{pmatrix} \right)=-4 \Im \Tr(x_2\bar{y}_2).
$$
Hence $\Xi_* \alpha=-2s_+$. 
\end{rem}

\begin{rem}\label{1 2 banach lg}
$i)$ We will need another complex Banach $*$-algebra: 
$$
\cB_{1,2}(\fH \oplus \fH, D):=\left\{   \begin{pmatrix}   x_{11}  &  x_{12} \\  x_{21}  &   x_{22} \end{pmatrix} \in \cB(\fH \oplus \fH) : x_{11}, x_{22} \in \cB_1(\fH), \, x_{12}, x_{21} \in \cB_2(\fH)   \right\},
$$  
which is endowed with the norm
$$
\|  X \|_{1,2}:=2 \max \{  \|x_{11}\|_1 , \, \|x_{22}\|_1 , \,  \|x_{12}\|_2 , \, \| x_{21}  \|_2  \}.
$$
This norm satisfies $\| XY\|_{1,2} \leq  \|X\|_{1,2} \|Y\|_{1,2}$. The involution $*$ is the usual adjoint of operators.
We remark that this type of block decompositions in terms of trace-class and Hilbert-Schmidt operators  appeared 
  in the  \textit{anomaly-free orthogonal and symplectic groups}   \cite{DG}, and in the study of infinite dimensional Poisson geometry  related to the restricted Grassmannian \cite{BRT, ABT}. 

\medskip

\noi $ii)$ We introduce a real subspace of $\cB_{1,2}(\fH \oplus \fH, D)$. Set
\begin{align*}
(\ub)_* & :=  \ub \cap \cB_{1,2}(\fH \oplus \fH, D) \\
& = \left\{  \begin{pmatrix}   x_1   &   x_2  \\   \bar{x}_2  &   \bar{x}_1    \end{pmatrix} :   x_1^*=-x_1 \in \cB_1(\fH), \, x_2^T=-x_2 \in \cB_2(\fH)    \right\}.
\end{align*}
The notation is justified by the following fact: the pairing
$$
\PI{ \, \cdot  \,}{ \, \cdot \,}: ( \ub)_*     \times  \ub  \to  \R, \, \, \, \, \, \PI{X}{Y}:=\Re \Tr(XY) 
$$
induces a topological isomorphism $((\ub)_*)^* \cong  \ub$. 
This is indeed an immediate consequence of \cite[Prop. 2.1]{BRT}, where it was proved that
$$
\mathfrak{u}_{\mathrm{res}}:=\left\{   \begin{pmatrix}   x_{11}   &   x_{12}  \\  - x_{12}^*   &   x_{22}    \end{pmatrix} \in \cB(\fH \oplus \fH)  : x_{11}=-x_{11}^*, \, x_{22}=-x_{22}^*, \,  x_{12} \in \cB_2(\fH)           \right\}
$$
and 
$$
( \mathfrak{u}_{\mathrm{res}})_*:=\left\{   \begin{pmatrix}   x_{11}   &   x_{12}  \\  - x_{12}^*   &   x_{22}    \end{pmatrix} \in \cB(\fH \oplus \fH)  :  x_{11}=-x_{11}^* \in \cB_1(\fH), \, x_{22}=-x_{22}^* \in \cB_1(\fH), \,   x_{12} \in \cB_2(\fH)           \right\}
$$
satisfy $(( \mathfrak{u}_{\mathrm{res}})_*)^*\cong \mathfrak{u}_{\mathrm{res}}$, where the duality pairing is given by $\PI{X}{Y}=\Tr(XY) = \Re \Tr(XY) $. In order to prove the topological isomorphism $(( \ub)_*)^* \cong  \ub$,  given $f \in (( \ub)_*)^*$, one extends this functional by the Hahn-Banach theorem to a functional in $(( \mathfrak{u}_{\mathrm{res}})_*)^*$. Therefore there exists $Y \in  \mathfrak{u}_{\mathrm{res}}$ such that $f(X)=\Re\Tr(XY)$, for all $X \in ( \ub)_*$. We only have to show that the element $Y$ can be changed for another element belonging to $ \ub$. To this end note that for any $X \in ( \ub)_*$, one has that $f(X)=f(IXI)=\Re \Tr((IXI) Y)=\Re \Tr(X (IYI))$. Thus, $f(X)=f(\frac{X+IXI}{2})=\Re \Tr(X(\frac{Y + IYI}{2}))$, where  $\frac{Y + IYI}{2} \in  \ub$.
\end{rem}




We now show that g1-pdms associated to quasi-free states give continuous real 2-cocycles of $\ub$. 




\begin{lem}\label{lema 1 sympl}
For $\Gamma \in  \DF$,  let $s_\Gamma:\ub \times \ub \to \R$ be defined by
$$
s_\Gamma(X,Y):=   \Tr(X[i\Gamma,Y])=\Tr(X\delta_\Gamma(Y)).
$$ 
Then $s_\Gamma$ is a continuous real 2-cocycle of $\ub$, which is cohomologous to $-s_+$. 
\end{lem}
\begin{proof}
 We first prove that $s_\Gamma(X,Y)<\infty$. If $X=\begin{pmatrix}
x_1 & x_2 \\
\bar{x}_2 & \bar{x}_1
\end{pmatrix} \in \ub$, 
$Y= \begin{pmatrix}
y_1 & y_2 \\
\bar{y}_2 & \bar{y}_1
\end{pmatrix} \in \ub
$, $\Gamma=\Gamma[\gamma, \alpha]$, 
then
$$
[\Gamma,Y]=
\begin{pmatrix}  
[\gamma,y_1] + \alpha \bar{y_2}-y_2 \alpha^*     &    \gamma y_2- y_2(\d1 - \bar{\gamma})+ \alpha \bar{y}_1 - y_1 \alpha  \\
\alpha^*y_1 - \bar{y}_1\alpha^* + (\d1 -\bar{\gamma})\bar{y}_2 - \bar{y}_2\gamma       &     \alpha^*y_2- \bar{y}_2\alpha   + [\d1 - \bar{\gamma},\bar{y}_1]
\end{pmatrix}:=\begin{pmatrix}  z_1    &   z_2  \\   - \bar{z}_2   &    - \bar{z}_1  \end{pmatrix},
$$
where $z_1 \in \cB_1(\fH)$, $z_2 \in \cB_2(\fH)$. In fact, the form of the right-hand matrix is derived by using that $\alpha^*=-\bar{\alpha}$. 
Then, it follows that
\begin{align*}
s_\Gamma(X,Y) & = i (\Tr(x_1z_1) - \Tr(\bar{x}_1\bar{z}_1)) + i(\Tr(\bar{x}_2 z_2)-\Tr(x_2 \bar{z}_2)) \\
& = - 2\Im \left( \Tr(x_1z_1)  + \Tr(\bar{x}_2 z_2) \right) <\infty
\end{align*}
because $x_2, z_2 \in \cB_2(\fH)$, $z_1 \in \cB_1(\fH)$ and $x_1 \in \cB(\fH)$, which implies $\bar{x}_2 z_2, x_1 z_1 \in \cB_1(\fH)$. The above expression also shows that $s_\Gamma(X,Y) \in \R$.
Let us show that $s_\Gamma$ is a continuous 2-cocycle. For we write 
$$
\Gamma=\begin{pmatrix} \gamma & \alpha   \\   -\bar{\alpha}   &  - \bar{\gamma}  \end{pmatrix}+ \begin{pmatrix} 0  &   0   \\   0   & \d1  \end{pmatrix}=:\Gamma_0 + P_-\, ,
$$
and 
\begin{equation*}
s_\Gamma(X,Y)= \Tr(X[i\Gamma_0 , Y])   +   \Tr(X[iP_- , Y])=:s_{\Gamma_0}(X,Y) +  s_{-}(X,Y)\,.
\end{equation*}
We  estimate the first term
\begin{align}
|s_{\Gamma_0}(X,Y)|  & = 
\left| \Tr \left( \begin{pmatrix}  x_1    &   x_2  \\    \bar{x}_2   &     \bar{x}_1  \end{pmatrix}  
\begin{pmatrix}  
[\gamma,y_1] + \alpha \bar{y_2}-y_2 \alpha^*     &    \gamma y_2+ y_2  \bar{\gamma}+ \alpha \bar{y}_1 - y_1 \alpha  \\
\alpha^*y_1 - \bar{y}_1\alpha^*   -\bar{\gamma}\bar{y}_2 - \bar{y}_2\gamma       &     \alpha^*y_2- \bar{y}_2\alpha   - [ \bar{\gamma},\bar{y}_1]
\end{pmatrix}  \right) \right| \nonumber  \\
& \leq   4  (   \| \gamma \|_1   +  \| \alpha \|_2   ) \|X\|_{\mathrm{res}}  \|Y\|_{\res} \nonumber \\
& \leq 4 \| \Gamma_0 \|_{1,2}      \|X\|_{\mathrm{res}}   \|Y\|_{\mathrm{res}} \, .  \label{estimate 12 res}
\end{align}
The second term can be estimated as follows
\begin{align*}
|s_-(X,Y)|  
  =   2 \left| \Im \Tr(\bar{x}_2 y_2)     \right|   \leq 2 \|x_2 \|_2 \|y_2\|_2 \leq  \|X\|_{\mathrm{res}}   \|Y\|_{\mathrm{res}} \, .  
\end{align*}
From these estimates we obtain that $s_\Gamma$ is continuous.  
 
We claim that $s_{\Gamma_0}$ is a coboundary. Notice that the space of finite-rank operators is dense in the Banach algebra $\cB_{1,2}(\fH\oplus \fH, D)$ given in Remark \ref{1 2 banach lg}. 
For $\Gamma_{00}$ a finite-rank operator, it clearly holds that $\Tr((X\Gamma_{00})Y)=\Tr(Y(X\Gamma_{00}))$, which implies that $s_{\Gamma_{00}}(X,Y)=-\Tr([X,Y]i\Gamma_{00}):=f_{\Gamma_{00}}([X,Y])$. Similar to the above estimates, one can see that $f_{\Gamma_0}(Z)=-\Tr(Z i \Gamma_0)$, where $\Gamma_0 \in \cB_{1,2}(\fH\oplus \fH, D)$, satisfies $|f_{\Gamma_0}(Z)| \leq  \|\Gamma_0\|_{1,2} \|Z\|_{\res}$. 
Combining this with \eqref{estimate 12 res}, we can take the limit of finite-rank operators in the norm $\| \, \cdot \, \|_{1,2}$ to get that $s_{\Gamma_0}(X,Y)=f_{\Gamma_0}([X,Y])$, which proves our claim. 

A simple computation shows that $s_-=-s_+$. Since $s_{\Gamma_0}$ is a coboundary and $s_+$ is a  2-cocyle, it follows that $s_{\Gamma}$ is a  2-cocyle. This also shows that that $s_{\Gamma}$ and $-s_+$ are  cohomologous.    
\end{proof}


\begin{lem}\label{lema 2 symplectic}
$
\ubg=\{ X \in \ub   :  s_\Gamma(X,Y)=0, \, \forall Y \in \ub \, \}.
$
\end{lem}
\begin{proof}
Straightforward computations show the following: 
$$
s_\Gamma(\Ad_V(X),\Ad_V(Y))=s_{\Ad_{V^*(\Gamma)}}(X,Y) \, ,  \, \, \, \, \, \, V \in \UB, \, X,Y \in \ub  .
$$
So we may assume that $\Gamma$ is a diagonal operator.
Thus,  to prove one inclusion, we take 
$$
X =  \begin{pmatrix}  x_1   &   x_2 \\  \bar{x}_2  &  \bar{x}_1 \end{pmatrix} \in \ubg, \, \,  \, \, \, Y=\begin{pmatrix}
 y_1 & y_2 \\
\bar{y}_2 &  \bar{y}_1
\end{pmatrix} \in \ub, \, \,  \, \, \,   \Gamma=\begin{pmatrix}   \Lambda   &   0   \\  0  &  \d1 - \Lambda  \end{pmatrix}.
 $$ 
We are going to show that $s_\Gamma(X,Y)=0$. 
We denote by $\lbrace \cdot , \cdot \rbrace$  the anticommutator of operators. 
Then,
\begin{align*}
s_\Gamma(X,Y) & = i \Tr \left(    
\begin{pmatrix}  x_1   &   x_2 \\  \bar{x}_2  &  \bar{x}_1 \end{pmatrix}   
\begin{pmatrix}
 [\Lambda , y_1] & \lbrace  \Lambda , y_2 \rbrace - y_2   \\
  \bar{y}_2   -   \lbrace  \Lambda  ,  \bar{y}_2 \rbrace  & [\bar{y}_1, \Lambda ]
\end{pmatrix}     
\right) \\
& = -2 \Im \left( \Tr(x_1 [\Lambda,y_1]) +   \Tr(x_2(\bar{y}_2 - \lbrace  \Lambda   , \bar{y}_2  \rbrace )) \right). 
\end{align*}
Now we use the characterization of the Lie algebra $\ubg$ proved in \eqref{2nd charac}. Since $\Lambda \in \cB_1(\fH)$ and $x_1 \Lambda=\Lambda x_1$, it follows that the first term vanishes. The second term also vanishes using that $x_2(\d1 - \Lambda) =\Lambda x_2$. 
This proves $s_\Gamma(X,Y)=0$. 

In order to show the reversed inclusion, we take an operator $X \in \ub$ such that $s_\Gamma(X,Y)= 0$ for all $Y \in \ub$.
As usual, we write
$$
X =  \begin{pmatrix}
 x_1 & x_2 \\
\bar{x}_2 & \bar{x}_1
\end{pmatrix}, \, \, \, \, \, \,  Y =  \begin{pmatrix}
 y_1 & y_2 \\
\bar{y}_2 & \bar{y}_1
\end{pmatrix},
$$ 
with $x_1^*=-x_1$, $-x_2^T = x_2 \in \cB_2(\fH)$, $y_1^* = -y_1$ and $-y_2^T = y_2 \in \cB_2(\fH)$. By the same calculation as above, 
we have
\begin{equation}\label{ig to be zero}
0=s_\Gamma(X,Y)=-2 \Im \left( \Tr(x_1 [\Lambda,y_1]) +   \Tr(x_2(\bar{y}_2 - \lbrace  \Lambda   , \bar{y}_2  \rbrace )) \right). 
\end{equation}
First, we take $y_2=0$, which gives
$
0=\Im  \Tr(x_1 [\Lambda,y_1])=\Im \Tr([x_1, \Lambda] y_1) 
$
for every $y_1^* = -y_1$. Put $z:=[x_1 ,\Lambda]$, which satisfies $z=z^*$. Take $y_1=iz$ to get $0=\Im \Tr(ziz)=\|z\|_2^2$, and consequently, 
$z=0$. That is, $x_1 \Lambda=\Lambda x_1$. 

Now we  have that \eqref{ig to be zero} turns into $0=\Im  \Tr(x_2(\bar{y}_2 - \lbrace  \Lambda   , \bar{y}_2  \rbrace ))=\Im \Tr((x_2(\d1-\Lambda)-\Lambda x_2)\bar{y}_2)$, for all $-y_2^T=y_2 \in \cB_2(\fH)$. Note that $z:=x_2(\d1-\Lambda)-\Lambda x_2$ satisfies $-z^T=z \in \cB_2(\fH)$. Then, take $y_2=iz$ to find that $0=\Im \Tr(z (-i\bar{z}))=\Im i\Tr(zz^*)=\|z\|_2^2$. Thus, $\Lambda x_2=x_2(\d1-\Lambda)$. 
Hence $X \in \ubg$. 
\end{proof}

Now we write the symplectic form of Theorem \ref{symplectic form construction} in our context. We take $M=\cO(\Gamma)$, $G=\UB$, $K=\UB^\Gamma$ and $s=s_\Gamma$ in the mentioned theorem. Recall that we may identify $T_\Gamma\cO(\Gamma)=\{ [X,\Gamma] : X \in \ub \}$. So  at the point $\Gamma$, set 
$$
\omega_\Gamma([X, \Gamma] , [Y,\Gamma]):= s_\Gamma(X,Y), \, \, \, \, \, X, Y \in \ub.
$$
Then, for $U \in \UB$,  
$$
\omega_{U\Gamma U^*}([X,U\Gamma U^*],[Y,U\Gamma U^*]):= \omega_\Gamma ([U^*XU,\Gamma],[U^*YU,\Gamma]), \, \, \,  \, \, X,Y \in \ub.
$$
Thus, $\omega=\Sigma(s_\Gamma)$ is defined by $\omega:=\{ \omega_{U\Gamma U^*}\}_{U \in \UB}$.

\begin{teo}\label{symplectic structure}
Let $\Gamma \in \DF$. Then $(\cO(\Gamma), \Sigma(s_\Gamma))$ is a weakly symplectic homogeneous space. Furthermore, if the spectrum of $\Gamma$ is finite, then $(\cO(\Gamma),\Sigma(s_\Gamma))$ is a strongly symplectic homogeneous space.
\end{teo}
\begin{proof}
All the necessary conditions to apply Theorem \ref{symplectic form construction}  have been proved in Theorem \ref{g1pdm smooth homog spaces}, and Lemmas  \ref{lema 1 sympl} and \ref{lema 2 symplectic}. Hence  $(\cO(\Gamma), \Sigma(s_\Gamma))$ becomes a weakly symplectic homogeneous space.

In order to show the second assertion, notice that by the invariance of the symplectic form it is enough to prove that for $X \in \ub$ the injective map 
$$
T_\Gamma \cO(\Gamma) \to T_\Gamma ^*\cO(\Gamma), \, \, \, \, \, [X,\Gamma ] \mapsto \Tr(X[i\Gamma, \, \cdot \, ]),
$$ 
is also surjective. In view of  \eqref{tangent space as range case 1} and \eqref{tangent space as range case 2} we note that 
$\ran(\delta_\Gamma)=T_\Gamma \cO(\Gamma) \subseteq i (\ub)_* \subseteq \cB_{1,2}(\fH \oplus \fH , D)$ if $\Gamma$ has finite spectrum. Since $\|x\| \leq \|x\|_1$ for any operator $x$, we derive the following relations between norms on the tangent space: for $X \in \ub$,
$$
\|[X,\Gamma] \|_{\Gamma}=\inf \{ \| X + Y\|_{\res}  : Y\Gamma = \Gamma Y   \} \leq \| [X, \Gamma] \|_{\res} \leq \| [X,\Gamma] \|_{1,2} \, .
$$ 
Now take $f \in T_\Gamma ^* \cO(\Gamma)$. From the previous relations between the norms, it follows that $f$ must be also continuous with respect to the norm $\| \, \cdot \, \|_{1,2}$. By the Hahn-Banach theorem, there exists a functional $\tilde{f} \in (i (\ub)_*)^*$ such that
$\tilde{f}|_{T_\Gamma \cO(\Gamma)}=f$. According to Remark \ref{1 2 banach lg} $ii)$, we have $(i (\ub)_*)^* \cong i \ub$, so that there is an operator $Z_0 \in i \ub$ such that $\tilde{f}=\Tr(Z_0  \, \cdot \,)$. Therefore, we get
$f([Y,\Gamma])= \tilde{f}([Y,\Gamma])  =\Tr( Z_0 [Y, \Gamma])=\Tr(iZ_0 [i \Gamma, Y])$, where $i Z_0 \in \ub$.
\end{proof}

\subsection{Orbits of g1-pdms as K\"ahler homogeneous spaces}\label{Complex struct}







We now show that orbits of g1-pdms are K\"ahler homogenous spaces of $\UB$ using a characterization of these spaces in Lie algebraic terms.
Thus we will construct  K\"ahler polarizations. We consider the finite and infinite spectrum cases separately.    

\begin{rem}
$i)$ Notice that $(\ub)_\C=\mathfrak{gl}_2(\fH \oplus \fH,I,D)=:\fG$ (see \eqref{complexif Lie alg}), meanwhile $X \mapsto \overline{X}=-X^*$ is the involution on $\fG$ whose set of fixed points is $\ub$. 
\end{rem}

\begin{rem}
 In the following lemma we assume that $\Gamma$ is expressed as 
\begin{equation}\label{diag form kahler}
\Gamma=\begin{pmatrix}  \Lambda  & 0  \\   0  &   \d1  - \Lambda \end{pmatrix}, \, \, \, \, \, \, \, \,  \Lambda=\sum_{i=1}^{r} \lambda_i p_i \, ,
\end{equation}
where $\{ p_i\}_{i=1}^{r+1}$ ($r < \infty$) is a decomposition of the identity with the properties stated in Remark \ref{proj ubog diag}   for $i=1, \ldots, r$, but we  modify the notation $p_{r+1}:=p_0$  and $\lambda_{r+1}:=\lambda_0=0$. In particular, recall that each $p_i$ is the projection onto the finite-dimensional subspace $\ker(\Lambda - \lambda_i \d1)$ for $i=1, \ldots, r$, and $p_{r+1}$ is the projection onto the infinite-dimensional subspace $\ker(\Lambda)$.  
 Further, we suppose that the eigenvalues satisfy $\lambda_i > \lambda_j$ if $i <j$, for all $i,j=1, \ldots, r+1$.  
\end{rem}

\begin{lem}[Finite spectrum]\label{fin spec polar}
Suppose that $\Gamma=\Gamma[\Lambda,0] \in \DF$ has finite spectrum and $\Gamma$ can be written as in \eqref{diag form kahler}. Then the following hold:
\begin{itemize}
\item[i)] If $\frac{1}{2} \notin \sigma(\Gamma)$, then 
$$
\fP:=\left\{  \begin{pmatrix} x   &   0      \\  y       &  - x^T          \end{pmatrix} \in \fG :   p_i x p_j=0 \text{ if }  i=1, \ldots, r+1, \, j <i     \right\}
$$
is a strongly K\"ahler polarization of $\ub$ in $s_\Gamma$. 
\item[ii)] If $\frac{1}{2} \in \sigma(\Gamma)$, then 
$$
\fP:=\left\{           \begin{pmatrix} x   &   z      \\  y       &  - x^T          \end{pmatrix} \in \fG :  p_i x p_j=0 \text{ if }  i=1, \ldots, r+1, \, j <i , \, z= p_1 z p_1               \right\}
$$
is a strongly K\"ahler polarization of $\ub$ in $s_\Gamma$. 
\end{itemize} 
\end{lem}
\begin{proof}
$i)$ We say that  an operator $x$ is upper triangular with respect to the decomposition of the identity $\{ p_i\}_{i=1}^{r+1}$ if $ p_i x p_j=0$  if   $i=1, \ldots, r+1$, $j <i$. Observe that  that $\fP$ is a closed subalgebra of $\fG$ since it is clearly a complex subspace, and
$$
\left[\begin{pmatrix}  x_1   &   0  \\  y_1   &   -x_1^T \end{pmatrix},   \begin{pmatrix}  x_2   &   0  \\  y_2   &   -x_2^T \end{pmatrix} \right] 
=\begin{pmatrix}  [x_1 , x_2]  &   0  \\  w   &   -[x_1,x_2]^T \end{pmatrix}, 
$$
where $[x_1, x_2]$ is upper triangular with respect to the decomposition of the identity $\{ p_i\}_{i=1}^{r+1}$ if $x_1$, $x_2$ are, and $w=y_1 x_2 - x_1^Ty_2 - y_2 x_1 + x_2^Ty_1$ satisfies $w^T=-w \in \cB_2(\fH)$ if $y_i^T=-y_i \in \cB_2(\fH)$, $i=1,2$.  Also it is clear that $\fP$ is closed in $\fG$, and $\fP$ admits the  closed complement 
$$
\fN:=\left\{  \begin{pmatrix} x   &   z      \\  0       &  - x^T          \end{pmatrix} \in \fG :   p_i x p_j=0 \text{ if }  i=1, \ldots, r+1, \, j \geq i    \right\}.
$$
Recall the description of $\ubg$ in the proof of Lemma \ref{lemma Lie subg g1-pdms}, which after taking into account the notation change $p_0:=p_{r+1}$, leads to 
$$
(\ubg)_\C=\left\{  \begin{pmatrix} x    &   0     \\   0     &   -x^T  \end{pmatrix} \in \fG : x=\sum_{i=1}^{r+1}p_i xp_i    \right\}.   
$$
Straightforward computations show that $\fP + \overline{\fP}=\fG$, $\fP \cap \overline{\fP}=(\ubg)_\C$ and $\Ad_U\fP \subseteq \fP$ for all $U \in \UB^\Gamma$ (see \eqref{isot group sin medio}).

Now take 
$$
X_i=\begin{pmatrix}  x_i   &   0  \\  y_i   &   -x_i^T \end{pmatrix} \in \fP, \, \, \, \, i=1,2.
$$
Recall that $\Tr(x)=\Tr(x^T)$ for any $x \in \cB_1(\fH)$. By this fact we obtain
\begin{align*}
s_\Gamma(X_1,X_2)& =i \Tr ( X_1[\Gamma, X_2]) = i \Tr \left( \begin{pmatrix} x_1[\Lambda,x_2]   &   0   \\  *   &   -x_1^T[\Lambda,x_2^T] 
 \end{pmatrix} \right) \\
& = i \Tr(x_1[\Lambda , x_2] -  x_1^T[\Lambda, x_2^T]) \\
& =i ( \Tr(x_1 \Lambda x_2 - x_2 \Lambda x_1)=i \sum_{i=1}^r \lambda_i (\Tr(p_ix_1p_ix_2p_i) - \Tr(p_ix_2p_ix_1p_i))=0.   
\end{align*}
Hence $s_\Gamma(\fP \times \fP)=0$. 

Let us show the last condition:
$$
-i s_\Gamma(X,\overline{X})>0, \, \, \, \, \, X=\begin{pmatrix} x   &   0      \\  y       &  - x^T          \end{pmatrix} \in \fP \setminus (\ubg)_\C \, .
$$
To this end, recall that $\Lambda=\sum_{i=1}^{r+1}\lambda_i p_i=\sum_{i=1}^r \lambda_i p_i$, and notice
\begin{align*}
-i s_\Gamma(X,\overline{X}) & =-\Tr( X [\Gamma, X^*]) 
= - \Tr \left(  \begin{pmatrix} x   &   0      \\  y       &  - x^T          \end{pmatrix}   
\begin{pmatrix} [\Lambda,x^*]   &   \Lambda y^*  - y^* (\d1 - \Lambda)      \\  0       &  [\Lambda,\bar{x}]          \end{pmatrix}    \right) \\
& = -\Tr \left( \begin{pmatrix}  x[\Lambda, x^*]   &  *     \\ *   &  y\Lambda y^* - yy^* (\d1 - \Lambda) - x^T[\Lambda, \bar{x}] \end{pmatrix} \right) \\
& = - 2 \Tr(x [\Lambda, x^*]) + \Tr(yy^*- (y^*y+yy^*)\Lambda):= A + B.
\end{align*}
In the last line we have used that $\Tr(x [\Lambda, x^*]) =-\Tr(x^T[\Lambda, \bar{x}])$. 
Now observe $x$ is upper triangular with respect to $\{ p_i \}_{i=1}^{r+1}$ if and only if $x=\sum_{i=1}^{r+1} \sum_{j \geq i} p_i x p_j$. 
Then, we have
\begin{align*}
A & = - 2 \Tr(x [\Lambda, x^*]) = 2 \Tr((xx^*- x^*x) \Lambda) \\
& = \Tr\left(\sum_{i=1}^{r} \sum_{j \geq i} \lambda_i p_i x p_j x^* p_i\right) - \Tr\left(\sum_{i=1}^{r+1} \sum_{j \geq i} \lambda_j p_j x^* p_i x p_j\right)  \\
& = 2 \sum_{i=1}^{r} \sum_{j > i} (\lambda_i - \lambda_j) \|p_i x p_j\|_2^2 .
\end{align*}
Since  $\lambda_1 > \ldots >\lambda_r>\lambda_{r+1}=0$, it follows from this expression  that $A>0$ when $p_i x p_j \neq 0$ for some $i=1, \ldots , r$, $j=i+1, \ldots , r+1$. 

On the other hand, the second term can be estimated as follows
\begin{align*}  
B & =\Tr(yy^*- (y^*y+yy^*)\Lambda) = \|y\|_2^2 - \sum_{i=1}^r (\lambda_i \Tr(y^*y p_i) +  \lambda_i \Tr(yy^* p_i)) \\
& \geq \|y\|_2^2 - \lambda_1 \left(\sum_{i=1}^r \Tr(p_iy^*yp_i) +   \Tr(p_iyy^*p_i) \right) \\
& \geq (1 - 2 \lambda_1 )\| y \|_2^2 .
\end{align*}
The last inequality is a consequence of the fact $\| \sum_{i=1}^r p_i z p_i  \|_1 \leq \| z \|_1 $, for all $z \in \cB_1(\fH)$ (see e.g. \cite[Chap. III Thm. 4.2]{GK60}) and  $\| \sum_{i=1}^r p_i z p_i  \|_1=\Tr(\sum_{i=1}^r p_i z p_i)$ for $z \geq 0$. Notice that $B>0$ if $y \neq 0$ because $\lambda_1=\max_{i=1, \ldots , r} \lambda_i < \frac{1}{2}$. 
We conclude that $-i s(X,\overline{X})>0$ whenever $X \in \fP \setminus (\ubg)_\C$, which occurs when $p_i x p_j \neq 0$ for some $i=1, \ldots , r$, $j=i+1, \ldots , r+1$, or when $y \neq 0$. Hence $\fP$ is a  polarization of $\ub$ in $s_\Gamma$, which is strong by Theorem \ref{symplectic structure}. 

\medskip

\noi $ii)$ The definition of $\fP$ is changed for this case where $\frac{1}{2} \in \sigma(\Gamma)$. In order to see that $\fP$ is a subalgebra of $\fG$, we consider the four entries $[X_1, X_2]_{ij}$ of the following operator:
$$
[X_1, X_2]= \begin{pmatrix} x_1   &   z_1      \\  y_1       &  - x_1^T          \end{pmatrix} 
\begin{pmatrix} x_2   &   z_2      \\  y_2       &  - x_2^T          \end{pmatrix}  -   
\begin{pmatrix} x_2   &   z_2      \\  y_2       &  - x_2^T          \end{pmatrix}       
\begin{pmatrix} x_1   &   z_1      \\  y_1       &  - x_1^T          \end{pmatrix}. 
$$
First, we see that
$$
[X_1,X_2]_{11}=[x_1,x_2] + p_1(z_1p_1y_2 - z_2p_1y_1),
$$
where both terms are clearly upper triangular with respect to the decomposition $\{ p_i\}_{i=1}^{r+1}$. Next note 
\begin{align*}
[X_1,X_2]_{12} & =x_1 z_2 - z_1 x_2^T - (x_2 z_1 - z_2 x_1^T) \\
& =p_1(x_1p_1 z_2 + x_2p_1 z_1 - z_1p_1 x_2^T  - z_2 p_1 x_1^T )p_1 :=p_1 w p_1,
\end{align*}
where $w^T=-w$. The next term is given by
$$
[X_1,X_2]_{21}=y_1 x_2 - x_1^Ty_2 - (y_2x_1 - x_2^Ty_1).
$$
Here we observe that $[X_1,X_2]_{21} \in \cB_2(\fH)$ because $y_i \in \cB_2(\fH)$, and $[X_1,X_2]_{21}^T=- [X_1,X_2]_{21}$.
The last term satisfies
$$
[X_1,X_2]_{22}=[x_1^T, x_2^T] + (y_1p_1 z_2 - y_2 p_1 z_1)p_1=-[X_1,X_2]_{11}^T.
$$
This completes the proof that $\fP$ is a subalgebra of $\fG$. Clearly, a closed supplement of $\fP$ in $\fG$ is given by
$$
\fN:=\left\{  \begin{pmatrix} x   &   z      \\  0       &  - x^T          \end{pmatrix} \in \fG :   p_i x p_j=0 \text{ if }  i=1, \ldots, r+1, \, j \geq i, \, \text{ and } p_1zp_1=0      \right\}.
$$
In this case, the complexification of $\ubg$ turns out to be
\begin{equation}\label{compl ubg sec case}
(\ubg)_\C=\left\{  \begin{pmatrix} x    &   z     \\   y     &   -x^T  \end{pmatrix} \in \fG : x=\sum_{i=1}^{r+1}p_i xp_i, \, z=p_1zp_1, \, z=-z^T, \, y=p_1yp_1, \, y=-y^T    \right\}.   
\end{equation}
Again the reader can check the three conditions $\fP + \overline{\fP}=\fG$, $\fP \cap \overline{\fP}=(\ubg)_\C$ and $\Ad_U\fP \subseteq \fP$ for all $U \in \UB^\Gamma$ (see \eqref{isot group sec case}). Now we show that $s_\Gamma(X_1,X_2)=0$, for all $X_1 , X_2 \in \fP$. 
For  notice
\begin{align*}
s_\Gamma(X_1, X_2) & = \Tr\left(  \begin{pmatrix} x_1   &   z_1      \\  y_1       &  - x_1^T          \end{pmatrix} 
\left[i \begin{pmatrix} \Lambda   &   0      \\  0       &  \d1 - \Lambda          \end{pmatrix}, 
\begin{pmatrix} x_2   &   z_2      \\  y_2       &  - x_2^T          \end{pmatrix}  \right] \right)  \\
& = i\Tr(x_1 [\Lambda , x_2] - x_1^T[\Lambda, x_2^T])  +   i\Tr(z_1( (\d1 - \Lambda)y_2 - y_2\Lambda)) + i\Tr(y_1(\Lambda z_2 - z_2(\d1 - \Lambda))) \\
& =0 
\end{align*}
This follows because the first term vanishes as in the previous case. For the  second term observe that $z_1=p_1z_1p_1$, where $p_1$ is the projection onto $\ker(\Lambda - \frac{1}{2})$. Therefore,
$$
\Tr(z_1( (\d1 - \Lambda)y_2 - y_2\Lambda))=\Tr(z_1 p_1((\d1 - \Lambda)y_2 - y_2 \Lambda)p_1)=0.
$$
For the third term, note $\Lambda z_2 - z_2(\d1 - \Lambda)=\frac{1}{2} (p_1 z_2p_1 - p_1 z_2p_1 )=0$. 

 In order to prove the last condition of a  K\"ahler polarization, we compute  for $X \in \fP \setminus (\ubg)_\C$
\begin{align*}
-i s_\Gamma(X,\overline{X}) & =-\Tr( X [\Gamma, X^*]) 
= - \Tr \left(  \begin{pmatrix} x   &   z      \\  y       &  - x^T          \end{pmatrix} 
\left[ \begin{pmatrix} \Lambda   &   0      \\  0       &  \d1 - \Lambda          \end{pmatrix} , 
\begin{pmatrix} x^*   &   y^*      \\  z^*      &  - \bar{x}          \end{pmatrix} \right] \right) \\
& = - 2 \Tr(x [\Lambda, x^*]) + \Tr(yy^*- (y^*y+yy^*)\Lambda) - \Tr(z((\d1 - \Lambda)z^*-z^*\Lambda)) \\
& := A + B + C.
\end{align*}
Again since $z=p_1 z p_1$, it follows that $C=0$. Using the description of $(\ubg)_\C$  in \eqref{compl ubg sec case}, we have that $X \in \fP \setminus (\ubg)_\C$ if and only if $p_i x p_j \neq 0$ for some $i=1, \ldots , r$, $j>i$, or when $y \neq 0$, $y \neq p_1 y p_1$. The term $A$ can be estimated as in the first case, but for the other term note $\|y\|_2^2=\|p_1 y \|_2^2 + \|(\d1 -p_1) y\|_2^2$, which yields 
\begin{align*}  
B & =\Tr(yy^*- (y^*y+yy^*)\Lambda) = \|(\d1-p_1)y\|_2^2 - \sum_{i=2}^r (\lambda_i \Tr(y^*y p_i) +  \lambda_i \Tr(yy^* p_i)) \\
& \geq \|(\d1 - p_1)y\|_2^2 - \lambda_2 \left(\sum_{i=2}^r \Tr(p_i(y(\d1-p_1))^*y(\d1-p_1)p_i) +   \Tr(p_i(\d1-p_1)y((\d1-p_1)y)^*p_i) \right) \\
& \geq (1 - 2 \lambda_2 )\| (\d1-p_1)y \|_2^2 ,
\end{align*}
where   $\lambda_2=\max_{i=2, \ldots , r} \lambda_i < \frac{1}{2}$ and $(\d1-p_1)y \neq 0$, for $y \neq p_1yp_1$. Hence
$-i s_\Gamma(X,\overline{X})>0$ for $X \in \fP \setminus (\ubg)_\C$. 
\end{proof}

Now we consider the infinite spectrum case.

\begin{rem}
  For following lemma we suppose that $\Gamma$ has infinite spectrum and $\ker(\Gamma)=\{0\}$. Thus, it can be written as  
\begin{equation}\label{diag form kahler2}
\Gamma=\begin{pmatrix}  \Lambda  & 0  \\   0  &   \d1  - \Lambda \end{pmatrix}, \, \, \, \, \, \, \, \,  \Lambda=\sum_{i=1}^\infty \lambda_i p_i \, ,
\end{equation}
where $\lambda_i \neq 0$, $i \geq 1$, and  $\{ p_i\}_{i=1}^\infty$ is a decomposition of the identity with the properties stated in Remark \ref{proj ubog diag}. Thus, each $p_i$ is the projection onto the finite-dimensional subspace $\ker(\Lambda - \lambda_i \d1)$ for $i \geq 1$. We again  suppose that the eigenvalues satisfy $\lambda_i > \lambda_j$ if $i <j$, for all $i,j \geq 1$.  
\end{rem}

\begin{lem}[Infinite spectrum]\label{inf spec polar}
Suppose that $\Gamma=\Gamma[\Lambda,0] \in \DF$ has infinite spectrum, $\mathrm{ker}(\Gamma)=\{ 0\}$ and $\Gamma$ can be expressed as in \eqref{diag form kahler2}.   Then the following hold:
\begin{itemize}
\item[i)] If $\frac{1}{2} \notin \sigma(\Gamma)$, then 
$$
\fP:=\left\{  \begin{pmatrix} x   &   0      \\  y       &  - x^T          \end{pmatrix} \in \fG :   p_i x p_j=0 \text{ if }  i \geq 1, \, j <i     \right\}
$$
is a weakly K\"ahler polarization of $\ub$ in $s_\Gamma$. 
\item[ii)] If $\frac{1}{2} \in \sigma(\Gamma)$, then 
$$
\fP:=\left\{           \begin{pmatrix} x   &   z      \\  y       &  - x^T          \end{pmatrix} \in \fG :  p_i x p_j=0 \text{ if }  i \geq 1, \, j <i , \, z= p_1 z p_1               \right\}
$$
is a weakly K\"ahler polarization of $\ub$ in $s_\Gamma$. 
\end{itemize} 
\end{lem}
\begin{proof}
$i)$ The proof follows the same lines as the finite-dimensional case with the following additional remarks. For a decomposition of the identity $\{ p_i \}_{i \geq 1}$ and  $x \in \cB(\fH)$, it is well-known that $x=\sum_{i,j\geq 1}p_ixp_j$, where the convergence of this series is in the WOT (weak operator topology). Furthermore, $\|\sum_{i=1}^N\sum_{j=1}^Mp_ixp_j\| \leq \|x\|$ for all $M,N\geq 1$. 
Notice that $x$ is then upper triangular with respect to $\{ p_i \}_{i \geq 1}$ if and only if $x=\sum_{i=1}^{\infty} \sum_{j \geq i} p_i x p_j$. From these facts we may write $x=\sum_{i=1}^{\infty} \sum_{j \geq i} p_i x p_j + \sum_{i=1}^{\infty} \sum_{j < i} p_i x p_j$, where both series converge in the WOT. Thus, the following subspace
$$
\fN:=\left\{  \begin{pmatrix} x   &   z      \\  0       &  - x^T          \end{pmatrix} \in \fG :   p_i x p_j=0 \text{ if }  i\geq 1, \, j \geq i       \right\}
$$
becomes a closed supplement for $\fP$ in $\fG$. Using that the adjoint and multiplication by a fixed operator are both continuous maps in the WOT, all but the last of the conditions of a  K\"ahler polarization are straightforward to check.

The last condition reads $-i s_\Gamma(X,\overline{X})=- 2 \Tr(x [\Lambda, x^*]) + \Tr(yy^*- (y^*y+yy^*)\Lambda)>0$ for all $X \in \fP \setminus (\ubg)_\C$. Following the finite spectrum case treated in Lemma \ref{fin spec polar}, in order to deal with the first term of $-i s_\Gamma(X,\overline{X})$ we need to compute $\Tr((xx^*-x^*x)\Lambda)$, where 
$x=\sum_{i=1}^{\infty} \sum_{j \geq i} p_i x p_j$ is now a WOT-convergent series whose partial sums are bounded in norm. To handle 
$\Tr(xx^*\Lambda)$ we apply twice the fact that functionals of the form $f(y)= \Tr( y z)$, $z \in \cB_1(\fH)$, $y \in \cB(\fH)$, are WOT continuous on the unit ball (\cite[Thm. 5.3]{EK98}). That is, we have 
\begin{align*}
\Tr(xx^* \Lambda) & = \sum_{i=1}^{\infty} \sum_{j \geq i} \Tr(p_ixp_j x^*\Lambda) \\
& = \sum_{i=1}^{\infty} \sum_{j \geq i} \Tr\left(p_ixp_j \left(\sum_{k=1}^{\infty} \sum_{l \geq k} p_lx^*p_k\right)\left(\sum_{s=1}^\infty \lambda_s p_s\right)\right) \\
& = \sum_{i=1}^{\infty} \sum_{j \geq i} \lambda_i \Tr(p_ixp_j x^* p_i).
\end{align*}
The term $\Tr(x^*x\Lambda)$ can be handled analogously. Hence, 
$$
- 2 \Tr(x [\Lambda, x^*]) =2 \sum_{i=1}^{\infty} \sum_{j > i} (\lambda_i - \lambda_j) \|p_i x p_j\|_2^2 .
$$
Since $\ker(\Gamma)=\{0\}$ and $\lambda_1 > \lambda_2 > \ldots $, we conclude that $\lambda_i - \lambda_j>0$ for all $j>i$. The second term of $-i s_\Gamma(X,\overline{X})$ can be treated as in the finite spectrum case. 
Finally, we observe that Theorem \ref{symplectic structure} implies now that $\fP$ is weakly polarization because the spectrum of $\Gamma$ is infinite.

\medskip

\noi $ii)$ Similar remarks apply to this case. 
\end{proof}




The previous results now imply the following.

\begin{teo}\label{kahler structures}
Let $\Gamma \in \DF$. Then the following assertions hold:
\begin{itemize}
\item[i)] If $\Gamma$ has finite spectrum, then $(\cO(\Gamma),\Sigma(s_\Gamma))$ is a strongly K\"ahler homogeneous space.
\item[ii)] If $\Gamma$ has infinite spectrum and $\mathrm{ker}(\Gamma)=\{0\}$, then $(\cO(\Gamma),\Sigma(s_\Gamma))$ is a weakly K\"ahler homogeneous space.
\end{itemize}
\end{teo}
 \begin{proof}
According to Theorem \ref{ubog diagonalization} there exists a diagonal g1-pdm $\Gamma=\Gamma[\Lambda, 0]$ satisfying the conditions of Lemmas \ref{fin spec polar} or \ref{inf spec polar}. Then the assertion of each item $i)$ or $ii)$ follows by Theorem \ref{kahler en terminos de alg Lie}. 
\end{proof}

\begin{rem}
Each orbit $\cO(\Gamma)$ can be endowed with a complex manifold structure  such that   the maps $\alpha_U: \cO(\Gamma) \to \cO(\Gamma)$, $\alpha_U(\Gamma_1)=U\Gamma_1U^*$, $\Gamma_1 \in \cO(\Gamma)$,  are holomorphic for every $U\in \UB$. Indeed, this follows by noting that the requirements of  Remark \ref{complex} have been be checked in Lemmas \ref{fin spec polar} and \ref{inf spec polar} without any assumption on the kernel of $\Gamma$.
\end{rem}




\appendix

\section{Appendix}\label{apendice hfb banach manifolds}

\subsection{Hartree-Fock-Bogoliubov  theory}\label{HFB}

In this appendix we provide the necessary background on Hartree-Fock-Bogoliubov (HFB) theory. Most of the material  is presented following \cite{BBKM14, BLS94, VFJ09}, where the reader can find the proofs omitted here. We  prefer the basis-dependent approach considered in these references, which has  the benefit of dealing with linear operators. A basis independent approach, where one deals with  conjugate-linear operators, is also used  in the literature (see, e.g., \cite{S07}).





\medskip

\noi \textit{Fock space formalism.}  
The one-particle Hilbert space $\fH$ is a complex separable Hilbert space endowed with an inner product $\PI{ \, }{ \, }$, which is linear in the second variable and conjugate-linear in the first. The fermionic Fock space $\cF$ over $\fH$ is defined as the orthogonal sum
$$
\cF:=\cF[\fH]=\bigoplus_{N=0}^{\infty} \cF^{(N)},
$$
where $\cF^{(N)}=\bigwedge_{n=1}^N \fH$ is the $N$-fold  skew-symmetric tensor product of $\fH$ for $N\geq 1$; and $\cF^{(0)}=\C \Omega$, where $\Omega=(1,0,0, \ldots) \in \cF$ is the  vacuum vector. Thus, each vector in $\cF^{(N)}$  is a possibly infinite combination of simple vectors of the form
$$
f_1 \wedge \ldots \wedge f_N := \frac{1}{\sqrt{N!}}\sum_{ \pi \in S_N} \text{sign}(\pi) f_{\pi(1)} \otimes f_{\pi(2)} \otimes \ldots \otimes f_{\pi(N)}, \, \, \, \,
f_i \in \fH,
$$
where $S_N$ denotes the symmetric group and $\text{sign}(\pi)$ the signature of the permutation $\pi$. 
Notice that the inner product between simple vectors on each $\cF^{(N)}$ is given by
$$
\PI{f_1 \wedge \ldots \wedge f_N}{g_1 \wedge \ldots \wedge g_N}_{\cF^{(N)}}=  \det(\PI{f_i}{g_j})_{1 \leq i,j \leq  N}, \, \, \, \, f_i, g_i \in \fH. 
$$
The Fock space $\cF$ becomes a Hilbert space equipped with the inner product defined by 
$$
\PI{\Psi_1}{\Psi_2}_\cF:=\sum_{N=0}^\infty\PI{f_1^{(N)}}{f_2^{(N)}}_{\cF^{(N)}}, \, \, \, \, \Psi_i=(f_i^{(0)}, f_i^{(1)}, \ldots), \, \, \,  i=1,2.
$$
We now introduce some useful linear operators on the Fock space. 
Associated to each vector $f \in \fH$, there are  creation operators $c^*(f)$ and annihilation operators $c(f)$  acting on $\cF$. The creation operator is defined on simple vectors by 
$$
c^*(f)\Omega=f; \, \, \, \, \, \,
 c^*(f)(f_1 \wedge \ldots \wedge f_N):=f \wedge f_1 \wedge \ldots \wedge f_N.
$$
Then, $c^*(f)$ extends to $\cF$ by linearity. The annihilation operator $c(f)$ is defined as the adjoint $c(f):=(c^*(f))^*$. 
It is noteworthy that $c^*(f)$ and $c(f)$ are linear and conjugate-linear in $f$, respectively. 
Further, we note $c(f)\Omega=0$. Using that $f_1 \wedge \ldots \wedge f_N = c^*(f_1) c^*(f_2) \ldots c^*(f_N)\Omega$, it follows that any vector of $\cF$ can be expressed as a limit of polynomials in creation operators acting on the vacuum vector. Creation and annihilation operators satisfy the \textit{canonical anticommutation relations (CARs):}
\begin{align*}
& \{c(f), c^*(g)\}=\PI{f}{g} \d1_\cF ; \\
& \{c^*(f), c^*(g)\}=\{c(f), c(g)\}=0.  
\end{align*}
for every  $f,g \in \fH$. Here $\{ A,B\}:=AB+BA$ is the anticommutator of  $A$ and $B$. 
Let $\cB(\cF)$ denote the algebra of bounded linear operators on $\cF$. 
In particular, the CARs imply that $c(f), c^*(f) \in \cB(\cF)$. Moreover, $c(f)$ and $c^*(f)$ are a partial isometries on $\cF$, and the map $c$  isometric, i.e. $\|c(f)\| = \|f\|$ and $\|c^*(f)\|=\|f\|$.
The unital $C^*$-algebra generated by $\{ c(f) : f \in \fH \}$ in $\cB(\cF)$ is a concrete construction of the \textit{CAR algebra} $\cA=\cA(\fH)$ of $\fH$. 


Another important operator on the Fock space is the \textit{particle number operator} $\widehat{\dN}$, which is defined by
$$
\widehat{\dN}:= \bigoplus_{N \geq 0} N \d1_{\cF^{(N)}}.
$$ 
This is an unbounded operator, its maximal domain is given by
those vectors
$\Psi= (f^{(0)}, f^{(1)}, \ldots )$ such that $\sum_{N=1}^\infty N^2 \|  f^{(N)}\|^2_{\cF^{(N)}}< \infty$.

\medskip

\noi \textit{Bogoliubov transformations.}  Throughout, we fix an orthonormal basis $\{ \varphi_k\}_{k \geq 1}$ of $\fH$.
 Associated to this basis, we define a  conjugation on $\fH$ by 
\begin{equation}\label{J complex c}
I_0 f:=\bar{f}:=\sum_{k \geq 1} \bar{\mu}_k \varphi_k
\end{equation}
if $f=\sum_{k \geq 1} \mu_k \varphi_k$. By a conjugation we mean that $I_0$ is a conjugate-linear isometry acting on $\fH$ such that $I_0^2=\d1$.  Then, for any operator $x$, we define the operators 
$$
\bar{x}:=I_0 x I_0 \, , \, \, \, \, \, \text{and} \, \, \, \, \, \,  x^T:=(\bar{x})^*=\overline{x^*}.
$$ 

Let $\cB(\fH)$ be the algebra of bounded operators on $\fH$. Consider two operators $u,v \in \cB(\fH)$ satisfying the following relations
\begin{align}
& uu^* + vv^*=\d1 =u^*u + v^T\bar{v}; \label{Bog 1}  \\
& u^*v + v^T\bar{u}=0=uv^T + v u^T \label{Bog 2}. 
\end{align}
Notice that equations \eqref{Bog 1} and \eqref{Bog 2} are equivalent to say that the following operator
\begin{equation}\label{Bog}
U=\begin{pmatrix}  u   &   v \\ \bar{v}  &   \bar{u}  \end{pmatrix}
\end{equation}
is unitary on $\fH \oplus \fH$. Unitary operators on $\fH \oplus \fH$ having this form are known as \textit{Bogoliubov transformations}.

For each Bogoliubov transformation $U$  as in \eqref{Bog},  the operators $d^*(f):=c^*(uf)+ c(v\bar{f})$, $d(f):=(d^*(f))^*$, also fulfill the CARs. 
Then there is a unique $*$-automorphism
$\beta_U$ of the CAR algebra  such that $\beta_U(c(f))=d(f)$, for all $f \in \fH$ (\cite[Thm. 5.2.5]{BR81}).   
It is of interest to know whether there exists  
a unitary operator that implements the automorphism $\beta_U$ on the Fock space. 
There is a well-known characterization  (see \cite{A70}):

\begin{teo}\label{bog trans and unit impl}
Let $U$ be a Bogoliubov transformation as in  \eqref{Bog}. Then there exists a unitary operator $\dU:=\dU_U:\cF \to \cF$ such that 
\begin{equation}\label{fock space}
\dU c^*(f) \dU^*= c^*(uf) + c(v\bar{f}), 
\end{equation}
for all $f \in \fH$, if and only if $\Tr(v^*v)<\infty$.
\end{teo}

Any such a unitary operator $\dU$ is called a \textit{unitary implementer of $U$}, or simply an \textit{implementer of $U$}; it is uniquely determined up to a scalar factor of modulus one.  The condition of the operator $v$ being a Hilbert-Schmidt operator is usually refer to as the \textit{Shale-Stinespring condition}. The group of all unitary implementers  is also known as the current group, and it is topological group in the strong operator topology (see \cite{A87}).



\medskip

\noi \textit{Quasi-free states.} A state $\o$ on the  algebra $\cB(\cF)$ is a complex linear map such that $\o(\d1_\cF)=1$ and $\o(A^*A)\geq 0$, for all $A \in \cB(\cF)$. In particular, this implies that $\o$ is continuous.  
A state is normal if there exists a density matrix $\rho$ (i.e. $\rho$ is a trace-class operator on $\cF$, $\rho \geq 0$ and $\Tr(\rho)=1$) such that 
$\o(A)=\Tr_\cF(\rho A)$, for all $A \in \cB(\cF)$. Important examples of normal states for us are the pure states, which have the form   $\o(A)=\PI{\Psi}{A \Psi}_\cF$, for all $A \in \cB(\cF)$, and some unit vector $\Psi \in \cF$.



We are interested in the following special class of states. A normal state $\o$ is said to be \textit{quasi-free} if it satisfies Wick's theorem, i.e. given vectors $h_1, \ldots , h_{2n} \in \fH$ and $2n$ creation and annihilation operators $e_1, \ldots, e_{2n} \in \{ \, c^*(h_1), \, c(h_1), \, \ldots, \, c^*(h_{2n}), \,  c(h_{2n}) \, \}$, we have $\o(e_1 \ldots e_{2n-1})=0$ and
$$
\o(e_1 \ldots e_{2n})=\sum_{\pi \in P_{2n}} \text{sign}(\pi) \, \o(e_{\pi(1)}e_{\pi(2)})\ldots \o(e_{\pi(2n-1)}e_{\pi(2n)}),
$$
where $P_{2n}$ denotes the set of all permutations $\pi \in S_{2n}$ such that $\pi(1)<\pi(3) < \ldots < \pi(2n-1)$ and $\pi(2j-1)< \pi(2j)$, for all $1 \leq j \leq n$. For instance, when $n=2$, this means that $
\o(e_1e_2e_3e_4)=\o(e_1e_4)\o(e_2e_3) -\o(e_1 e_3)\o(e_2e_4) + \o(e_1e_2)\o(e_3e_4).
$
Additionally, we will often restrict to the following class of states of physical interest. A state $\o$ is said to have \textit{finite particle number} if 
\begin{equation}\label{def de oN}
\o(\widehat{\dN}):=\sum_{N=1}^\infty N \o(\Pi^{(N)})< \infty. 
\end{equation}
where $\Pi^{(N)}$ is the orthogonal projection onto $\cF^{(N)}$.

\begin{exa}
In the case in which $\o$ is a pure state associated to a vector $\Psi \in \cF^{(N)}$ for some $N \geq 0$, then $\o$ is a quasi-free state
if and only if $\Psi=f_1 \wedge \ldots \wedge f_N$, for some orthonormal vectors $f_1,  \ldots , f_N$ (take $\Psi=\Omega$ when $N=0$). Vectors 
 having this form are known as  \textit{Slater determinants}. Clearly, the corresponding pure state has finite particle number. 
\end{exa}




\begin{nota}
We will use the following notation:
\begin{align*}
\cZ & :=\{ \o : \o \text{ is a normal state} \}; \\
\Z & :=\{ \o \in \cZ : \o \text{ is a quasi-free state with finite particle number} \}.
\end{align*}
\end{nota}




\medskip

\noi \textit{Generalized one-particle density matrices.} Each state $\o$ gives raise to a \textit{generalized one-particle density matrix (g1-pdm)} $\Gamma_\o$, which is defined as the self-adjoint operator on $\fH\oplus \fH$ satisfying
$$
\PI{\begin{pmatrix} f_1 \\  f_2 \end{pmatrix}}{\Gamma_\o \begin{pmatrix} g_1 \\ g_2\end{pmatrix}}_{\fH\oplus \fH}=\o([c^*(g_1)+c(\bar{g}_2)][c(f_1)+c^*(\bar{f}_2)]), \, \, \, f_i \,, \, g_i \in \fH. 
$$
 Since the conjugation $I_0$ might change (or not) when the basis $\{ \varphi_k \}_{k \geq 1}$ changes, the definition of $\Gamma_\o$ depends on the fixed basis. A straightforward computation using the CARs shows that 
\begin{equation}\label{low 1}
0 \leq \Gamma_\o \leq \d1_{\fH\oplus \fH} .
\end{equation}
It is convenient to express $\Gamma_\o$ as a $2 \times 2$ matrix of operators on $\fH\oplus \fH$,
\begin{equation}\label{adm op0}
\Gamma_\o=\begin{pmatrix} \gamma_\o   & \alpha_\o  \\  \alpha^*_\o  &  \d1 - \bar{\gamma}_\o  \end{pmatrix} =:\Gamma_\o[\gamma_\o, \alpha_\o].
\end{equation}
Recalling that $\{ \varphi_k\}_{k \geq 1}$ is our fixed orthonormal basis of $\fH$, we abbreviate $c_k^*:=c^*(\varphi_k)$ and $c_k:=c(\varphi_k)$.  The operator blocks of $\Gamma_\o$ can be expressed as 
\begin{equation}\label{adm op 0}
\PI{\varphi_m}{\gamma_\o \varphi_k}=\o(c_k^* c_m), \, \, \, \, \PI{\varphi_m}{\alpha^*_\o \varphi_k}=\o(c_k^* c_m^*). 
\end{equation}
Notice that 
\begin{equation}\label{adm op1}
\gamma^*_\o=\gamma_\o\, , \, \, \, \,  \alpha^T_\o=-\alpha_\o. 
\end{equation}
The operator $\gamma_\o$ is called \textit{one-particle density matrix (1-pdm)} of the state $\o$. In the special case of the pure state $\o_\Psi(A)=\PI{\Psi}{ A \Psi}_\cF$, for some unit vector $\Psi \in \cF^{(N)}$, the 1-pdm $\gamma_{\o_\Psi}$ satisfies $0 \leq \gamma_{\o_\Psi} \leq \d1$ and it  has trace equal to $N$. Conversely, Lieb \cite{L81} proved that for each self-adjoint operator $\gamma$ on $\fH$ such that $0 \leq \gamma \leq \d1$ and $\Tr(\gamma)=N$, there exists a unit vector $\Psi \in \cF^{(N)}$ such that $\gamma=\gamma_{\o_\Psi}$.   
The 1-pdm is related to the particle number operator. 
Noticing that $\widehat{\dN}=\sum_{k\geq 1} c^*_kc_k$ as a quadratic form, then
\begin{equation}\label{rel trac num}
\o(\widehat{\dN})=\sum_{k \geq 1} \o(c_k^* c_k) =  \Tr(\gamma_\o).
\end{equation}
It follows that $\gamma_\o$ is a trace-class operator if and only if $\Gamma_\o$ is the g1-pdm of a state $\o$ with finite particle number. 
The operator $\alpha_\o$ is called the \textit{pairing matrix}. When $\o(A)=\PI{\Psi}{A \Psi}$, for a vector $\Psi=(f^{(0)}, f^{(1)}, \ldots) \in \cF$, is a (pure) quasi-free state with finite particle number, 
then $\o(\widehat{\dN}^2)<\infty$, and its variance is expressed as
$$
\o(\widehat{\dN}^2) - \o(\widehat{\dN})^2= \sum_{n\geq 0} (n-N)^2\|f^{(N)}\|_{\cF^{(N)}}^2 =2 \Tr(\alpha_\o^* \alpha_\o).
$$
Thus, the Hilbert-Schmidt norm of $\alpha_\o$ measures the spreading of $\Psi$ among the subspaces $\cF^{(N)}$.

\medskip

 One can define generalized one-particle density matrices without reference to a  state, only taking into account the  properties that these operators have when they are associated to a state.
That is, a  g1-pdm can be defined as an operator of the form
 $$
\Gamma=\Gamma[\gamma, \alpha]= \begin{pmatrix} \gamma & \alpha   \\   \alpha^*   & \d1 - \bar{\gamma}  \end{pmatrix} \in \cB(\fH \oplus \fH)
 $$
 satisfying the obvious corresponding conditions \eqref{low 1}  and \eqref{adm op1}.
It is then natural to address the question of which g1-pdms are associated to a state. Using the relation in \eqref{rel trac num}, one can prove the following sufficient condition (\cite[Thm. 2.3]{BLS94}).



\begin{teo}\label{d trace} 
Let $\Gamma=\Gamma[\gamma , \alpha]$ be a g1-pdm such that $\Tr(\gamma)< \infty$. Then there exists a unique quasi-free state $\o$ with finite particle number such that $\Gamma=\Gamma_\o$.
\end{teo}

Any g1-pdm $\Gamma=\Gamma[\gamma, \alpha]$ such that $\Tr(\gamma)< \infty$ is called an \textit{admissible generalized one-particle density matrix}.

\begin{nota}
Let $\cB_1(\fH)$ and $\cB_2(\fH)$ be the ideals of trace-class and Hilbert-Schmidt operators on $\fH$, respectively.  The set of all admissible g1-pdms is denoted by
\begin{equation}\label{g1pdm}
\cD:=\left\{ \begin{pmatrix} \gamma & \alpha   \\   \alpha^*   & \d1 - \bar{\gamma}  \end{pmatrix} \in \cB(\fH \oplus \fH) :   0 \leq \begin{pmatrix} \gamma & \alpha   \\   \alpha^*   & \d1 - \bar{\gamma}  \end{pmatrix}    \leq \d1_{\fH\oplus \fH}, \, \gamma=\gamma^* \in \cB_1(\fH), \, \alpha^T=-\alpha   \right\}. 
\end{equation}
\end{nota}

\begin{rem}\label{d trace rem} 
 $i)$ Note that, by the previous result, there is a bijection between quasi-free states and the (convex) set $\DF$, namely $\Phi: \Z \to \DF$,   $\Phi(\o) = \Gamma_\o$.
Let $U:\fH \oplus \fH \to \fH \oplus \fH$ be a Bogoliubov transformation for which there exists a unitary implementer $\dU_U:\cF \to \cF$.
 For each state $\o \in \Z$, define $\o_U(A):=\o(\dU_U A \dU_U^*)$ for any $A \in \cA$. Using \eqref{fock space}, it easily follows that $\o_U \in \Z$. Furthermore, it can be shown that $\Phi(\o_U)=U^*\Gamma_\o U$.   

\medskip

\noi $ii)$  We claim that $\Gamma=\Gamma[\gamma, \alpha] \in \DF$ implies  $\alpha \in \cB_2(\fH)$. To see this, notice that $0\leq \Gamma \leq \d1_{\fH\oplus \fH}$ is equivalent to
$$
0 \leq \Gamma - \Gamma^2 = \begin{pmatrix} \gamma - \gamma^2 - \alpha \alpha^* & (\d1 - \gamma)\alpha + \alpha (\d1 - \bar{\gamma})   \\   
\alpha^*(\d1 - \gamma) +  (\d1 - \bar{\gamma})\alpha^*   &  \bar{\gamma} - \bar{\gamma}^2-\alpha^*\alpha  \end{pmatrix}
.
$$
In particular, we get that $\gamma \geq \gamma^2 + \alpha \alpha^*$, from which our claim follows by using  $\gamma \in \cB_1(\fH)$.
\end{rem}

We end this collection of results on g1-pdms with the well-known relation between pure states and projections, and a remarkable   diagonalization result (see \cite[Thm. 2.6]{BLS94} and the proof of \cite[Thm. 2.3]{BLS94}).

\begin{teo}\label{proj states}
A quasi-free state $\o$ with finite particle number is a pure state if and only if its g1-pdm $\Gamma_\o \in \cD$ is a projection on $\fH \oplus \fH$
(i.e. $\Gamma_\o^2=\Gamma_\o$).
\end{teo}

\begin{teo}\label{ubog diagonalization}
Let $\Gamma \in \cD$. Then there exists a Bogoliubov transformation $W$ satisfying the Shale-Stinespring condition such that
\begin{equation}\label{diagonal form}
W\Gamma W^*=\begin{pmatrix}  \Lambda   &   0  \\   0   &  \d1  - \Lambda \end{pmatrix},
\end{equation}
where $\Lambda$ is a diagonal operator with respect to the fixed orthonormal basis $\{ \varphi_k\}_{k\geq 1}$ satisfying
$\Tr(\Lambda) < \infty$ and $0 \leq \Lambda \leq \frac{1}{2}\d1$. 
\end{teo}


\medskip

\noi \textit{Hartree-Fock-Bogoliubov theory.} The Hamiltonian $\dH$ is a self-adjoint operator on the Fock space, which we additionally assume to be bounded from below. This Hamiltonian generates the dynamics of a  many-particle fermion system, that is, particles satisfying the Pauli exclusion principle. It is of relevance to compute  the total ground state energy in the grand canonical ensemble, which is given by the Rayleigh-Ritz principle:
$$
E_{gs}:=\inf\{  \o(\dH) : \o \in \cZ  \}.
$$ 
Notice that the expression $\o(\dH) \in [0, \infty]$  is well defined because $\dH$ is bounded from below (see \cite[Section 2.c]{BLS94}). This is indeed an extension of the original formulation of the Rayleigh-Ritz principle that only uses pure states for the computation of above infimum. Determining the ground state energy $E_{gs}$ is an almost impossible task (except in a few well-known examples). 
One of the most widely used approximation methods is the \textit{(traditional) Hartree-Fock method} in quantum chemistry to calculate the energy of atoms and molecules. This method replaces the set of all states for the set of all pure states given by Slater determinants:
\begin{equation*}
E_{HF}:=\inf\{ \PI{\Psi}{\dH \Psi}_\cF: N \in \dN, \, \Psi=f_1 \wedge \ldots \wedge f_N, \, \PI{f_i}{f_j}=\delta_{ij} \}.
\end{equation*}
In this case, one assumes that the Hamiltonian $\dH$ conserves the particle number (i.e. $[\dH,\widehat{\dN}]=0$). Notice the upper bound $E_{gs}\leq E_{HF}$.

In this work we are interested in the Hartree-Fock-Bogoliubov method, which generalizes the Hartree-Fock method and it relies on the machinery described above between Bogoliubov transformations, quasi-free states and generalized one-particle density matrices.  In this method the Hamiltonian is not assumed to conserve the particle number. The resulting  Hartree-Fock-Bogoliubov energy is obtained by restricting the  infimum to the set of all quasi-free
states: 
\begin{align*}
E_{HFB} & :=\inf\{  \o(\dH) : \o \in \Z  \}\\
& = \inf\{  \cE(\Gamma) : \Gamma \in \DF  \}.
\end{align*}
In the second line,  we have reformulated the infimum with the aid of the Hartree-Fock-Bogoliubov functional defined by $\cE(\Gamma_\o):=\o(\dH)$, 
where $\Gamma_\o$ is the g1-pdm of the quasi-free state $\o$. Observe again that $E_{gs}\leq E_{HFB}$. We are not dealing with domain requirements. The quasi-free states $\o$ for which the kinetic energy is finite (i.e. $\o(\dH)<\infty$) were studied for particular classes of Hamiltonians  (see \cite{VFJ09, LL10}).



Let us mention  examples of Hamiltonians that are usually considered. Let $h$ be  a self-adjoint operator on $\fH$, the one-body operator, and let $W$ be a self-adjoint operator  acting on $\fH \otimes \fH$, the two-body operator. We write $h_j$ for the operator $\d1 \otimes \ldots \d1 \otimes h \otimes \d1 \ldots \d1$, where $h$ is placed in the $j$th factor of the $N$-th fold tensor product. Similarly, $W_{ij}$ acts only the $i$th and $j$th factors. Then one considers a Hamiltonian $\dH$ on the Fock space formally defined by
$$
\dH=\bigoplus_{N\geq 0} \left( \sum_{i=1}^N   h_i   +   \sum_{1 \leq i, j \leq N}  W_{ij}   \right).
$$
The fermions to consider can be electrons, and $h$ the Hamiltonian from molecular physics given by taking $h=-\Delta -\sum_{k=1}^K \frac{Z_k}{|x - r_k|}- \mu$ acting on $\fH=L^2(\mathbb{R}^3, \C^2)$, and $W=\frac{1}{|x-y|}$ interpreted as a multiplication operator on $\fH \otimes \fH\cong L^2(\mathbb{R}^6,\C^2)$. Here $r_k \in \mathbb{R}^3$, $r=1, \ldots, K$, represents the positions of the $K$ nuclei, $Z_i>0$ are their charges, and $\mu <0$ is a sufficiently small chemical potential. Another example is the pseudo-relativistic Hamiltonian for describing fermions obtained by taking  $h=\sqrt{ - \Delta + m^2} -m$, $m >0$ is the mass, and $W$ the same as before (see \cite{LL10}). As a third example we refer to the Hubbard model treated in \cite{BLS94}, where the Hilbert space $\fH$ is now finite dimensional. 
In all cases, $\dH$ fulfills the condition of being bounded from below.

\subsection{Banach manifolds basics}\label{apendice}



We recall definitions and elementary properties of several geometric structures in the setting  of manifolds modeled on Banach spaces.
Basic definitions of notions such as Banach manifold, smooth or $C^\infty$ maps between manifolds, Banach-Lie group, exponential map, etc.  can be found in \cite{La95}.  The special setting of analytic Banach manifolds is treated in \cite{Up85}. For several geometric structures in Banach manifolds and their interplay with operator theory  we refer to \cite{B06}.   We  write  manifold for refer to a real or complex smooth manifold modeled on a  Banach space. Also we omit the terminology Banach manifold, Banach-Lie group, etc., which are replaced by manifold, Lie group, etc. 	 The tangent space of a manifold $M$ at a point $p$ (resp. cotangent space) is denoted by $T_p M$ (resp. $T_p^* M$). When $f:M \to N$ is a smooth map between two manifolds, we write $T_p f:T_p M \to T_{f(p)} N$ for the tangent of $f$ at $p$.   Given Lie groups $G, K, \ldots$, we use the notation $\mathfrak{g}, \mathfrak{k}, \ldots$ for its Lie algebras.   The exponential map of Lie group $G$ is denoted by $\exp_G$. 
	
\medskip


\noi \textit{Homogeneous spaces and submanifolds.} Let $M$ be a  manifold. A Lie group $G$ \textit{acts smoothly on $M$} if the there is a smooth   map $G \times M \to M$, $(g,p) \mapsto g \cdot p$, such that $(g h) \cdot p=g \cdot (h \cdot p)$, and $1 \cdot p=p$ for every $g, h \in G$ and $p \in M$. Notice that in this case the map $\pi_p:G \to M$, $\pi_p(g)=g \cdot p$ is also smooth. The orbit of a point $p \in M$  is defined by $\cO(p):=\{ g\cdot p : g \in G\}$. The action is said to be \textit{transitive} if $\cO(p)=M$ for some $p \in M$. The \textit{isotropy group at} $p \in M$ is $G_p:=\{ g \in G : g \cdot p = p\}$. Clearly, there is a bijection $\cO(p) \to G/G_p$, $g \cdot p \mapsto g G_p$. A smooth map $f:M \to N$ between two  manifolds is called a \textit{submersion at} $p \in M$ if $T_p f: T_p M \to T_{f(p)}N$ is surjective and $\ker(T_p f)$ is a complemented subspace in $T_p M$. In the case that $f$ is a submersion at every point of $M$, it is said that $f$ is a \textit{submersion}.  

Let $G$ be a Lie group acting smoothly and transitively on a manifold $M$. If there is some point $p \in M$ such that the map $\pi_p:G \to M$, $\pi_p(g)=g \cdot p$ is a submersion at $1 \in G$, then $M$ is called a \textit{smooth homogeneous space of} $G$.
We observe that the existence of such a point $p \in M$ in the above definition is actually equivalent to the fact that every point in $M$ have that property.

 Let $M$ be manifold and $N$ a subset of $M$.   $N$ is an \textit{embedded submanifold} of $M$ if for each point $p \in N$ there exists a Banach space $E$ and a chart $(\cW, \phi)$ at $p$,
 $\phi:\cW \subseteq E \to M$, $\phi(0)=p$, such that $\phi(\cW \cap (F \oplus \{ 0 \}) )=\phi(\cW) \cap N$, where $F$ is a closed and complemented subspace in $E$.  A smooth map $f:M \to N$ between two  manifolds is called an \textit{immersion at} $p \in M$ if $T_p f: T_p M \to T_{f(p)} N$ is injective and $\ran(T_p f)$ is a complemented closed subspace in $T_p M$. In the case that $f$ is an immersion at every point of $M$, it is said that $f$ is an \textit{immersion}. The following criterion is useful.

\begin{prop}\label{crit immersion}
Let $M$, $N$ be two manifolds, $N \subseteq M$. Then $N$ is an embedded submanifold of $M$ if and only if the topology of $N$ coincides with the topology inherited from $M$ and the inclusion map $N \hookrightarrow M$ is an immersion.
\end{prop}

The notion of Lie subgroup is related to the construction of homogeneous spaces. This notion is not uniform in the literature, see for instance \cite{N04} where various notions are recalled; and also some pathologies are pointed out, which only show up in this context of infinite dimensional Lie groups modeled on Banach spaces. In this work we use the most restrictive notion available following \cite{B67}: a subgroup $K$ of a Lie group $G$ is a \textit{Lie subgroup} if $K$ is an embedded submanifold of $G$. In particular, this allows us to state the next result as follows.

\begin{teo}\label{quotient Lie manifolds struct}
Let $G$ be a Lie group with Lie algebra $\fG$,  $K$ a closed subgroup of $G$, and set
$$ 
\nL(K) :=\{ X \in \fG :  \exp_G(tX) \in K, \, \forall t \in \R  \}.
$$
The following assertions hold:  
\begin{itemize}
\item[i)] $\nL(K)$ is a closed Lie subalgebra of $\fG$.
\item[ii)] If there are open neighborhoods $0 \in \cV \subseteq \fG$ and $\d1 \in \cW \subseteq G$ such that 
$\exp_G:\cV \to \cW$ is a diffeomorphism and $\exp_G(\cV \cap \nL(K))=\cW \cap K$, then $K$ carries a Lie group structure endowed with the relative topology inherited from $G$ such that $\nL(K)$ is the Lie algebra of $K$ and $\exp_K=\exp_G|_{\nL(K)}:\nL(K) \to K$.
\item[iii)] If, in addition, $\nL(K)$ is a complemented subspace in $\fG$, then $K$ is a Lie subgroup of $G$ and  the quotient space $M:=G/K$ carries the structure of a  manifold equipped with the quotient topology such that the 
natural projection $\pi: G \to M$, $\pi(g)=gK$, is a submersion. Furthermore, $G$ acts smoothly on $M$ by the left translation action defined by $ G \times M \to M$, $g \cdot hK=ghK$.
\end{itemize}
\end{teo}


We observe that the  manifold structure on $G/K$ is uniquely determined by the condition that the projection $\pi$ is a submersion.  Also the previous result has the following converse. A smooth homogeneous space $M$ of a Lie group $G$ must be diffeomorphic to $G/G_p$, where 
$G_p$ is  the isotropy group of $G$ at any point $p \in M$, and $G_p$ turns out to be a Lie subgroup of $G$. Let $\fG_p$ be the Lie algebra of $G_p$. A \textit{reductive structure} for a  smooth homogeneous space $M$ of a Lie group $G$ is a smooth distribution of closed subspaces $\{\mathfrak{m}_p\}_{p \in M}$ of $\fG$ such that $\fG_p \oplus \mathfrak{m}_p=\fG$ and $\Ad_g ( \mathfrak{m}_p )=\mathfrak{m}_p$, for every $g \in G_p$ and $p \in M$. A smooth homogeneous space equipped with a reductive structure is called a \textit{reductive homogeneous space}.


\medskip

\noi \textit{Symplectic homogeneous spaces.}  A \textit{weakly symplectic manifold} is a pair $(M,\omega)$ of a manifold $M$ and a non-degenerate closed  differential 2-form $\omega$ taking values on $\R$. 
It is said to be a \textit{strongly symplectic manifold} if for every $p \in M$ the injective map $T_p M \to T_p^* M$, $v \to \omega_p(v, \cdot)$, is also surjective.  Notice that the notions of weakly and strongly symplectic manifolds clearly coincide for finite dimensional manifolds. In the case in which $M=G/K$ is a smooth homogeneous space of $G$, consider the smooth map $\alpha_g: M \to M$ defined by $\alpha_g(hK)=ghK$. A 2-form $\omega$ is called \textit{invariant}  if $(\alpha_g)^*\omega=\omega$ for all $g \in G$. Here $(\alpha_g)^*\omega$ indicates  the pull-back of $\omega$ by $\alpha_g$. A \textit{weakly (resp. strongly) homogeneous space}
  is a weakly (resp. strongly) symplectic manifold  $(M, \omega)$, where  $M=G/K$ is a homogeneous space of $G$ and the form $\omega$ is invariant.

Although the forthcoming definitions can be given for a Banach space $\fZ$, we only assume that $\fZ=\R$ or $\fZ=\C$, which are the only cases needed in this work. 
 A \textit{continuous $\fZ$-valued 2-cocyle} of a Lie algebra $\fg$ is a  continuous skew-symmetric function $s:\fg \times \fg \to \fZ$
such that 
\begin{equation*}
s(X,[Y,Z]) + s (Z,[X,Y]) +  s(Y,[Z,X])=0,  
\end{equation*}
for all $X,Y,Z \in \fg$. It is called a \textit{coboundary} if there is a continuous linear map $f: \fg \to \fZ$ such that $s(X,Y)=f([X,Y])$. Let $Z^2_c(\fg, \fZ)$ and $B^2_c(\fg,\fZ)$ be the space of continuous $\fZ$-valued 2-cocycles and the subspace of coboundaries, respectively. The \textit{second continuous cohomology Lie algebra space} is defined as the following quotient of vector spaces $H^2_c(\fg, \fZ):=Z^2_c(\fg,\fZ)/B^2_c(\fg, \fZ)$. Two continuous $\fZ$-valued 2-cocycles are said to be \textit{cohomologous} if their classes coincide in $H^2_c(\fg, \fZ)$.

Let $G$ be a Lie group, $K$ a Lie subgroup of $G$ and $M=G/K$. Let $\fG$ and $\fK$ be the Lie algebras of $G$ and $K$, and $\pi: G \to G/K$ the natural projection, $p=\pi(1)$. Suppose that $s \in Z^2_c(\fg, \R)$ is such that $\fK=\{ X \in \fG :  s(X,Y)=0, \, \forall Y \in \fG   \}$. Since $T_p M \cong \fG / \fK$, and using the relation between $\fK$ and $s$,  the following non-degenerate skew-symmetric bilinear map is well defined: 
$$
\omega_p(X + \fK , Y + \fK):=s(X,Y), \, \, \, \, \, X, Y \in \fG.
$$
Then it can be smoothly translated to other points of $M$ by using the above defined maps $\alpha_g$, namely
$$
\omega_{g \cdot p}(v,w):=\omega_p(T_{g \cdot p} (\alpha_{g^{-1}})(v), T_{g \cdot p} (\alpha_{g^{-1}})(w) ), \, \, \, \, \, v,w \in T_{g \cdot p} M.
$$
We write $\omega=\Sigma(s)$ for the differential 2-form defined by $\omega:=\{ \omega_{g \cdot p}\}_{g \in G}$. 
For finite-dimensional Lie groups this construction provides an invariant symplectic form on their coadjoint orbits. 
We now state its generalization to the setting of Banach manifolds. 

\begin{teo}\label{symplectic form construction}
Let $G$ be a Lie group, $K$ a Lie subgroup of $G$ and $M=G/K$.   Assume that $s \in Z^2_c(\fg, \R)$ satisfies $\fK=\{ X \in \fG :  s(X,Y)=0, \, \forall Y \in \fG   \}$. Then $(M, \Sigma(s))$ is a weakly symplectic homogeneous space. 
\end{teo}

\medskip

\noi \textit{K\"ahler homogenous spaces.}  Let $G$ be a real Lie group and $K$ a Lie subgroup of $G$. Then, $M=G/K$ has the structure of real smooth homogenous space. Assume that $(M,\omega)$ is a weakly (resp. strongly) symplectic homogeneous space of $G$ satisfying the following conditions:
$M$ is actually a complex manifold and the  maps $\alpha_g:M \to M$, $\alpha_g(hK)=ghK$  are holomorphic for all $g \in G$; $\omega_p(v,w)=\omega_p(J_p v , J_p w)$ for all $v,w \in T_p M$ and $p \in M$, where $J_p$ is the complex structure on $T_p M$ giving the multiplication by i; and $\omega_p(v,J_p v)>0$ for all $v \in T_p M$, $v \neq 0$, and $p \in M$. In this case, $(M,\omega)$ is said to be \textit{weakly (resp. strongly) K\"ahler homogenous space of $G$}.

Next we describe how  K\"ahler homogenous space can be constructed in Lie algebraic terms. Given a real Lie group $G$ with Lie algebra $\fG$ we write $\fG_\C$ for the complexification of $\fG$, and we denote by $a \mapsto \overline{a}$
the involution of $\fG_\C$ whose set of fixed points is given by $\fG$.

\begin{rem}\label{complex}
The following criterion was proved in \cite{B05}. For real smooth homogeneous space $M=G/K$ there is a bijective correspondence between, on one hand, complex manifold structures on $M$ such that the smooth structure underlying the complex structure is just the smooth manifold structure of $M$ and  the maps $\alpha_g$, $g \in G$, are holomorphic, and on the other hand, closed complex subalgebras $\fP$ of $\fG_\C$ such that 
\begin{equation}\label{complex and p}
\fP + \overline{\fP} = \fG_\C, \, \, \, \, \fP \cap \overline{\fP}=\fK_\C, \, \, \, \, \text{and} \, \, \, \, \mathrm{Ad}_h \fP \subseteq \fP, \, \, h \in K.
\end{equation}
\end{rem}

\begin{teo}\label{kahler en terminos de alg Lie}
Let $G$ be a real Lie group, $K$ a Lie subgroup of $G$ and $M=G/K$ a smooth homogeneous space.   Let $s \in Z^2_c(\fg, \R)$ satisfying $\fK=\{ X \in \fG :  s(X,Y)=0, \, \forall Y \in \fG   \}$. Denote by $\omega=\Sigma(s)$, and consider the $\C$-linear extension $s:\fG_\C \times \fG_\C \to \C$. 
Moreover assume that $\fP$  is a complemented subalgebra of $\fG_\C$ with the properties stated in \eqref{complex and p},
and assume that $M$ has the structure of a complex manifold given by $\fP$. 
Then $(M,\omega)$ is a weakly K\"ahler homogenous space if and only if the following hold:
\begin{equation}\label{com pol}
s(\fP \times \fP)=0, \, \, \, \, \text{and} \, \, \, \, -is(a,\overline{a})>0, \, \, a \in \fP \setminus \fK_\C .
\end{equation}
Furthermore, if the injective map $\fG / \fK \to (\fG / \fK)^* \simeq T_p ^*M$, $X+ \fK \mapsto \omega_p(X + \fK , \, \cdot \,)$, is surjective, then $(M,\omega)$ is a strongly K\"ahler homogenous space.
\end{teo}

A complemented subalgebra $\fP$ of $\fG_\C$ satisfying the conditions in \eqref{complex and p} and \eqref{com pol} is called a \textit{weakly K\"ahler polarization} of $\fG$ in $s$. If the last condition in the previous theorem is also satisfied, then it is called a \textit{strongly K\"ahler polarization} of $\fG$ in $s$. For a proof of Theorem \ref{kahler en terminos de alg Lie} we refer to \cite{B06}.

\subsection*{Acknowledgment}
This research was partially supported by  CONICET (PIP 2021/2023 11220200103209CO), ANPCyT (2015 1505/ 2017 0883) and FCE-UNLP (11X974).


\bigskip

{\sc (Claudia D. Alvarado)} {Departamento de  Matem\'atica \& Centro de Matem\'atica La Plata, FCE-UNLP, Calles 50 y 115, 
(1900) La Plata, Argentina  and Instituto Argentino de Mate\-m\'atica, `Alberto P. Calder\'on', CONICET, Saavedra 15 3er. piso,
(1083) Buenos Aires, Argentina.}

\noi e-mail: {\sf calvarado@mate.unlp.edu.ar}


\bigskip

{\sc (Eduardo Chiumiento)} {Departamento de  Matem\'atica \& Centro de Matem\'atica La Plata, FCE-UNLP, Calles 50 y 115, 
(1900) La Plata, Argentina  and Instituto Argentino de Matem\'atica, `Alberto P. Calder\'on', CONICET, Saavedra 15 3er. piso,
(1083) Buenos Aires, Argentina.}     
                                               
\noi e-mail: {\sf eduardo@mate.unlp.edu.ar}


%
%

\end{document}